\documentclass[11pt,reqno]{amsart}

\usepackage{amsthm,  amsfonts, url, booktabs, tikz, setspace, fancyhdr, amsbsy}

\usepackage{amsmath,amsfonts,amssymb,mathrsfs}
\usepackage{amssymb,mathrsfs,graphicx}
\usepackage{epsfig,subfigure}
\usepackage{ifthen}
\usepackage[titletoc]{appendix}
\usepackage{colortbl}
\definecolor{black}{rgb}{0.0, 0.0, 0.0}
\definecolor{red}{rgb}{1.0, 0.5, 0.5}

\usepackage{ifthen} 

\newcommand{\bke}[1]{\left( #1 \right)}

\newcommand{\norm}[1]{\left\Vert #1 \right\Vert}
\newcommand{\abs}[1]{\left| #1 \right|}

\newcommand{\calX}{{ \mathcal X  }}

\provideboolean{shownotes} 
\setboolean{shownotes}{true} 

\newcommand{\hole}[1]{
\ifthenelse{\boolean{shownotes}}%
{\begin{center} \fbox{ \rule {.25cm}{0cm} \rule[-.1cm]{0cm}{.4cm}
\parbox{.85\textwidth}{\begin{center} \texttt{#1}\end{center}} \rule
{.25cm}{0cm}}\end{center}} {} }
\title[Regular solutions of non-Newtonian
fluids] {Existence of regular solutions for a certain type of
non-Newtonian fluids}

\author[Kyungkeun Kang]{Kyungkeun Kang}
\address[Kyungkeun Kang]{\newline Department of Mathematics \newline Yonsei University, Seoul 03722, Korea (Republic of)}
\email{kkang@yonsei.ac.kr}
\author[Hwa Kil Kim]{Hwa Kil Kim}
\address[Hwa Kil Kim]{\newline Department of Mathematics Education \newline Hannam University, Daejeon, 34430, Korea (Republic of)}
\email{hwakil@hnu.kr}
\author[Jae-Myoung Kim]{Jae-Myoung Kim}
\address[Jae-Myoung Kim]{\newline Center for Mathematical Analysis \& Computation \newline  Yonsei University, Seoul 03722, Korea (Republic of)}
\email{cauchy02@naver.com}

\numberwithin{equation}{section}

\newtheorem{theorem}{Theorem}[section]
\newtheorem{lemma}[theorem]{Lemma}

\newtheorem{corollary}[theorem]{Corollary}

\newtheorem{remark}[theorem]{Remark}

\newcommand{\bbr}{\mathbb R}

\newcommand{\R}{{\mathbb R }}

\def\charf {\mbox{{\text 1}\kern-.30em {\text l}}}

\def\Xint#1{\mathchoice
   {\XXint\displaystyle\textstyle{#1}}%
   {\XXint\textstyle\scriptstyle{#1}}%
   {\XXint\scriptstyle\scriptscriptstyle{#1}}%
   {\XXint\scriptscriptstyle\scriptscriptstyle{#1}}%
   \!\int}
\def\XXint#1#2#3{{\setbox0=\hbox{$#1{#2#3}{\int}$}
     \vcenter{\hbox{$#2#3$}}\kern-.5\wd0}}

\def\aint{\Xint\diagup}

\newenvironment{pfthm1}{{\par\noindent
            \textbf{Proof of Theorem \ref{main}}\quad}}{}

\newenvironment{pfmain15}{{\par\noindent
            \textbf{Proof of Corollary \ref{main-cor}}\quad}}{}

\newenvironment{pflem2}{{\par\noindent
            \textbf{Proof of Lemma \ref{deriv-G}}\quad}}{}

\newenvironment{pflem4}{{\par\noindent
            \textbf{Proof of Lemma \ref{deriv-G : time}}\quad}}{}

\begin{document}
\allowdisplaybreaks

\date{\today}



\begin{abstract}
We are concerned with existence of regular solutions for
non-Newtonian fluids in dimension three. For a certain type of
non-Newtonian fluids we prove local existence of unique regular
solutions, provided that the initial data are sufficiently smooth.
Moreover, if the $H^3$-norm of initial data is sufficiently small,
then the regular solution exists globally in time.
\end{abstract}

\maketitle

\centerline{\date}

\vspace{-5mm} \noindent {2000 AMS Subject Classification}: 76A05,
76D03, 49N60

%
%
%
%
\section{Introduction}
We consider the Cauchy problem of non-Newtonian fluids in three
dimensions
\begin{equation}\label{NNS}
\left\{
\begin{array}{ll}
\displaystyle u_t- \nabla \cdot\big( G[|Du|^2] Du \big) +(u \cdot
\nabla) u +\nabla p= 0,\\
\vspace{-3mm}\\
\displaystyle \text{div} \ u =0,
\end{array}\right.
\quad \mbox{ in } \,\,\R^3\times (0,\, T)
\end{equation}
with the viscous part of the stress tensor, $G[|Du|^2]$, such that
$G : [0,\infty) \rightarrow [0,\infty)$ satisfies for any
$s\in[0,\infty)$
\begin{equation}\label{G : property}
\begin{aligned}
&G[s] \geq m_0 >0, \qquad G[s]+ 2 G^{'}[s]s \geq m_0> 0,\\
&|G^{(k)}[s] s^\alpha| \leq C_k  |G^{(k-1)}[s]| \quad \alpha \in \{
0,1 \}.
\end{aligned}
\end{equation}
Here, $G^{(k)}[\cdot]$ is the $k-$th derivative of $G$, $m_0$ and
$C_k$ are positive constants, and $Du$ denote the symmetric part of
the velocity gradient, i.e.
\begin{equation*}\label{vis}
~~Du = D_{ij}u:= \frac{1}{2}\Big(\frac{\partial u_i}{\partial
x_j}+\frac{\partial u_j}{\partial x_i}\Big), \ i,j=1,2,3.
\end{equation*}
In \eqref{NNS}, $u:\R^3\times (0,\, T)\rightarrow\R^3$ and
$p:\R^3\times (0,\, T)\rightarrow\R$ represent the flow velocity
vector and the scalar pressure, respectively. We study Cauchy
problem of \eqref{NNS}, which requires initial conditions
\begin{equation}\label{ini}
u(x,0)=u_0(x), \qquad x\in\R^3, \qquad {\rm div}\,u_0=0.
\end{equation}
We note that a typical type of the viscous part of the stress tensor
satisfying the property \eqref{G : property} is of some power-law
models, for example,
$G[|Du|^2]=(m_0^{\frac{2}{q-2}}+|Du|^2)^{\frac{q-2}{2}}$ with
$2<q<\infty$ or $G[|Du|^2]=m_0+(\sigma+|Du|^2)^{\frac{q-2}{2}}$ with
$1<q<\infty$, $\sigma>0$.

It is said that a fluid is Newtonian if the viscous stress tensor is
 a linear function of the rate of deformation tensor (in this case $G[|Du|^2]$ is nothing but a constant). On
the contrary, for some fluids such as blood, paint and starch, it is
observed that the relation between the shear stress and the shear
strain rate is non-linear, and we commonly call those to the
non-Newtonian fluids (see e.g. \cite{AM74}, \cite{B}, \cite{M-R05}). 
In this paper, we establish the existence of a unique regular
solution for the incompressible non-Newtonian fluids \eqref{NNS},
\eqref{G : property} and \eqref{ini}, in particular, by estimating
higher derivatives in $H^l(\R^3)$, $l\ge 3$.

We report shortly some known results related to the existence of
solutions. In the case that $G[s] = (\mu_0 +
\mu_1s)^{\frac{q-2}{2}}$ with positive constants $\mu_0$ and
$\mu_1$, M\'alek, Ne\v cas, Rokyta and R$\stackrel{\circ}{\textrm
u}$\v zi\v cka proved in \cite{M-N-R-R} that a strong solution
exists globally in time in periodic domains for $q \geq
\frac{11}{5}$ in dimension three and for $q>1$ in dimension two,
respectively (see \cite{P} for the whole space case in dimension two
or three).

Also, they established local existence of strong solution in time
for $q>\frac{5}{3}$ in three dimensional periodic domains (refer to
\cite{B-D-R10} for shear thinning case, $\frac{7}{5}<q<2$). 
Here by strong solutions we mean solutions solving
the equations a. e. and satisfying the following energy estimate:
\begin{equation*}\label{strong-class}
\sup_{0\leq t\leq T}\|u(t)\|^2_{H^{1}(\bbr^3)} +\int_0^T
\|u(t)\|^2_{H^{2}(\bbr^3)}
 + \int_0^T \int_{\bbr^3}|Du|^{q-2}|\nabla Du|^2\, \leq
 C\|u_0\|^2_{H^{1}(\bbr^3)}.
\end{equation*}

For a bounded domain $\Omega\subset \R^3$ with
no-slip boundary condition, Amamm \cite{A94} proved that if initial
data are sufficiently regular and small, the unique strong solution
$u$ exists globally in the class $u \in C((0,T),
W^{2,q}(\Omega))\cap C^1((0,T),L^q(\Omega))$, $q>3$. Moreover, the
exponential decay of $u$ in time was obtained as well. For general
initial data, Bothe and Pr\"{u}ss \cite{B-P07} established local in
time well-posedness in the class $u \in L^p((0,T),
W^{2,p}(\Omega))\cap W^{1,p}((0,T),L^q(\Omega))$, $p>n+2$ based on
$L^p$-maximal regularity theory in a bounded smooth domain
$\Omega\subset \R^n$ with no-slip or slip boundary conditions (see
other papers e.g. \cite{B09-1}, \cite{B11}, \cite{BW} and therein
reference for related to strong solutions).


%
%
%

On the other hand, weak solutions are meant to solve the equations
in the sense of distributions and satisfy
\begin{equation*}\label{strong-class}
\sup_{0\leq t\leq T}\|u(t)\|^2_{L^{2}(\bbr^3)} +\int_0^T (\|\nabla
u(t)\|^2_{L^{2}(\bbr^3)}+\|\nabla u(t)\|^q_{L^{q}(\bbr^3)})\,dt \leq
C\|u_0\|^2_{L^{2}(\bbr^3)}.
\end{equation*}

In any dimension $n\ge 2$ and $\mu_0\ge 0$, the existence of weak
solutions was firstly shown in \cite{L67,L69,Lion69} for
$\frac{3n+2}{n+2}\leq q$, and later, the result was improved up to $
\frac{2n}{n+2}<q$ in \cite{D-R-W10} (see also \cite{W} and other
related references therein).

In \cite{K-M-S02}, Kaplick\'{y}, Malek and Stara considered
 non-Newtonian fluids with a stress
tensor of the form $2U^{\prime}[|Du|^2 ]Du_{ij},$
where
$U:[0,\infty)\rightarrow[0,\infty)$ is $C^2$-function such that
\begin{equation}\label{property-S-20}
\frac{\partial^2 U[|Du|^2]}{\partial D_{ij}\partial D_{kl}}
D_{ij}D_{kl}\geq C_1 (1+|Du|^2 )^{\frac{q-2}{2}}|Du|^2,\quad
|\frac{\partial^2U[|Du|^2 ]}{\partial D_{ij}\partial D_{kl}}|\leq
C_2 (1+|Du|^2 )^{\frac{q-2}{2}}.
\end{equation}
One typical example of the stress tensor satisfying above
assumptions is $(1+|Du|^2 )^{\frac{q-2}{2}}Du$ and the corresponding
potential $U$ becomes $\frac{1}{q} (1+|Du|^2 )^{\frac{q}{2}}$
(compared to our notation, $G[|Du|^2 ]=2U'[|Du|^2 ])$.
It was shown in \cite{K-M-S02} that in case that $q>\frac{4}{3}$,
$C^{1,\alpha}$ regularity solution exists for the non-Newtonian
fluid flows satisfying \eqref{property-S-20} in two dimensional
periodic domain $\mathbb{T}^2$. More specifically, the authors deal
with the equations \eqref{NNS} and \eqref{ini} involving the stress
tensor with \eqref{property-S-20} for the case of periodic domains
$\mathbb{T}^2$, and when $q>\frac{4}{3}$, they established the
global-in-time existence of a H\"{o}lder continuous solution,
namely,
%
$$
u\in C^{1,\alpha} (\mathbb{T}^2\times(0,T)) \quad \mbox{and}\quad
p\in C^{0,\alpha} (\mathbb{T}^2\times (0,T)), \qquad 0<\alpha< 1.
$$
(see \cite{K05} for a bounded domain
$\Omega\subset \R^2$ with Dirichlet boundary condition)

In case of three dimensions, it is, however, unknown
for the existence of $C^{k,\alpha}$, $k\ge 1$, solution. The method
of proof in \cite{K-M-S02} seems to work on only two dimensions and,
as far as we know, its extension to three dimensions has
not been made so far. Our primary objective of this paper is to
construct classical solutions rather than strong solutions for the
system \eqref{NNS}-\eqref{ini} in dimension three. Our main
results are two-fold. Firstly, if initial data $u_0$ belongs to
$H^l(\R^3)$ with $l\ge 3$, we establish a local regular solution for
some time $T_l$ in the class
\[
\calX_l([0, T_{l}];\bbr^3)\,:\,= L^{\infty}([0, T_{l}];\;
H^{l}(\bbr^3))\cap L^{2}([0, T_{l}]; \; H^{l+1}(\bbr^3)),
\]
and furthermore, such solution is unique (see Theorem \ref{main}). A
consequence for local existence of regular solutions is that
$\nabla^{l-2} u$ and $\partial_t^{\frac{l-2}{2}} u$ become H\"older
continuous for an even integer $l>3$ (see Corollary \ref{main-cor}).
Secondly, we can obtain a global regular solution, provided that
initial data is sufficiently small. To be more precise, if initial
data $u_0\in H^l(\R^3)$, $l\ge 3$ such that $\|u_0\|_{H^{3}(\R^3)}$
is sufficiently small, then the local solution in Theorem \ref{main}
exists in fact globally in time (see Theorem \ref{main1}).

One of main observations is that there are two good terms caused by
energy estimates of higher derivatives, provided that the condition
\eqref{G : property} is satisfied. More specifically, in case $l\ge
3$, 
we can see that the following two integrals appear
\begin{equation}\label{April23-1000}
\int_{\R^3}G[|Du|^2]|\partial^lD u|^2\,dx,\qquad
\int_{\R^3}G^{'}[|Du|^2]|Du:
\partial^lDu|^2\,dx.
\end{equation}
It turns out that, due to the hypothesis $G[s] \geq m_0$ and $G[s]+
2 G^{'}[s]s \geq m_0$ in \eqref{G : property}, the sum of two
integrals in \eqref{April23-1000} is bounded below by
$m_0\int_{\R^3}|\partial^lD u|^2\,dx$, which plays an important role
for local existence of solutions.

Now we are ready to state our first main result.

\begin{theorem}\label{main}
Let $u_0 \in H^l(\R^3)$, $l\geq 3$. There exists
$T_{l}:=T(\|u_0\|_{H^{l}})>0$ such that the equation
\eqref{NNS}-\eqref{ini} has a unique solution $u$ in the class
$\calX_l([0, T_{l}];\bbr^3)$. Furthermore, the solution $u$
satisfies
\begin{equation}\label{Oct17-2017-10}
\sup_{0\leq t\leq T_l}\|u(t) \|^2_{H^l}+\int_0^{T_l}\|u(t)
\|^2_{H^{l+1}} <C=C(\|u_0\|_{H^{l}}).
\end{equation}
\end{theorem}

A consequence of Theorem \ref{main} is the H\"older continuity of
the regular solutions until the time of existence.

\begin{corollary}\label{main-cor}
Let $l$ be an even positive integer. Under the assumption of Theorem
\ref{main}, the solution $u$ of\eqref{NNS}-\eqref{ini} belongs to
$C_x^{l-2,\frac{1}{2}}C_t^{\frac{l-2}{2},\frac{1}{4}}(\R^3\times
(0,T_l))$.
\end{corollary}

Another main result is the global existence of regular solutions, in
case that initial data are sufficiently small. More precisely, our
second result reads as follows:

\begin{theorem}\label{main1}
Let $u_0 \in H^l(\R^3)$, $l\geq 3$. There exists $\epsilon_0>0$ such
that if $\|u_0\|_{H^{3}(\R^3)}<\epsilon$ for any $\epsilon$ with
$0<\epsilon \leq \epsilon_0$, then the unique regular solution $u$
in Theorem \ref{main} is extended globally in time, i.e.
$T_l=\infty$.
\end{theorem}

We emphasize that global wellposedness of the regular
solution in Theorem \ref{main1} requires only smallness of
$H_3$-norm of initial data, not demanding control of the size of
$\|u_0\|_{H^{l}(\R^3)}$, $l>3$, from which a direct consequence is
the following:
\begin{corollary}\label{main1-cor}
Let $u_0 \in C^\infty(\R^3)$ be satisfying
$\|u_0\|_{H^{l}(\R^3)}<\infty$ for any $l\geq 3$. If
$\|u_0\|_{H^{3}(\R^3)}<\epsilon_0$, where $\epsilon_0$ is given in
Theorem \ref{main1}, then the unique smooth solution $u$ of
\eqref{NNS}-\eqref{ini} exists  globally in time and $u$ satisfies
the estimate \eqref{Oct17-2017-10}.
\end{corollary}

This paper is structured as follows: In Section 2, we state a key
lemma, whose proof is given in the Appendix. In Section 3, the proof
of Theorem \ref{main} is presented. Section 4 is devoted to proving
Corollary \ref{main-cor}. In Section 5, we provide the proof of
Theorem \ref{main1}.

\section{Preliminaries}

We introduce some notations. For $1\leq q\leq \infty$, we denote by
$W^{k,q}(\bbr^3)$ the standard Sobolev space.
 Let $(X,\|\cdot\|_{X})$ be a normed space and by $L^q(0,T;X)$ we mean the space of all
Bochner measurable functions $\varphi:(0,T)\rightarrow X$ such that
\begin{equation*}
\left\{
\begin{array}{ll}
\displaystyle\|\varphi\|_{L^q(0,T;X)}:=\Big(\int_0^T \|\varphi(t)\|^q_{X}dt\Big)^{\frac{1}{q}}<\infty \hspace{1.5cm}\mbox{if} \quad 1\leq q<\infty,\\
\vspace{-3mm}\\
\|\varphi\|_{L^{\infty}(0,T;X)}:=\displaystyle\mbox{\rm ess\,sup}_{t\in (0,T)}\|\varphi(t)\|_{X}<\infty \quad\quad\ \mbox{if} \quad  q=\infty.\\
\vspace{-3mm}
\end{array}\right.
\end{equation*}
We denote by $C^{\alpha}_xC^{\alpha/2}_t$ (or $C^{\alpha}_{x,t}$)
the space of H\"older continuous functions with an exponent
$\alpha\in (0, 1)$. For a non-negative integer $k$ we mean, in
general, by $C^{2k, \alpha}_xC^{k, \alpha/2}_t$ (or $C^{2k,
\alpha}_{x,t}$) the space of functions whose mixed derivative
$\nabla^{2(i-j)}_x\nabla^j_t u$ belongs to
$C^{\alpha}_xC^{\alpha/2}_t$ for all integers $i,j$ with $0\le j\le
i\le k$. Let $a_{ij}$ and $b_{ij}$ with $i,j=1,2,3$ be scalar
functions, and for $3\times 3$ matrices $A=(a_{ij})_{i,j=1}^3$ and
$B=(b_{ij})_{i,j=1}^3$ we write
$$ A : B = \sum_{i,j=1}^3 a_{ij}b_{ij},\qquad \nabla A : \nabla B =\sum_{i,j=1}^3 \nabla a_{ij} \cdot \nabla b_{ij},
\qquad \nabla^2 A : \nabla^2 B = \sum_{i,j=1}^3 \nabla^2 a_{ij} :
\nabla^2 b_{ij}. $$ The letter $C$ is used to represent a generic
constant, which may change from line to line.

Next lemma is a key observation for our analysis, which shows some
estimates of higher derivatives for the viscous part of the stress
tensor.

\begin{lemma}\label{deriv-G}
Let $l$ be a positive integer, $\tilde{\sigma}_l :\{1,2,\cdots,
l\}\rightarrow \{1,2,\cdots, l \}$ a permutation of $\{1,2,\cdots, l
\}$, and $\pi_l$ a mapping from $\{1,2,\cdots, l\}$ to $\{1,2,3\}$.
Suppose that $u\in C^{\infty}(\R^3)\cap H^l(\R^3)$. Assume further
that $G : [0,\infty) \rightarrow [0,\infty)$ is infinitely
differentiable and satisfies properties given in \eqref{G :
property}. Then, the multi-derivative of $G$ can be rewritten as the
following decomposition:
\begin{equation*}\label{eq-1 : derive G}
\partial_{x_{\sigma_l(l)}}\partial_{x_{\sigma_l(l-1)}}
\cdots\partial_{x_{\sigma_l(1)}}G[|Du|^2] = 2\big (G'[|Du|^2] Du:
\partial_{x_{\sigma_l(l)}}\partial_{x_{\sigma_l(l-1)}}
\cdots\partial_{x_{\sigma_l(1)}}D u \big) + E_l,
\end{equation*}
where $\sigma_l:=\pi_l\circ\tilde{\sigma}_l$ and
\begin{equation*}\label{eq-2 : derive G}
E_l=2\big(\partial_{x_{\sigma(l)}} (G^{'}[|Du|^2]Du):
\partial^{l-1} Du \big) +
\partial_{x_{\sigma(l)}} E_{l-1},\qquad \quad E_1=0,
\end{equation*}
where
$\partial^{l-1}:=\partial_{x_{\sigma(l-1)}}\cdots\partial_{x_{\sigma(1)}}$.
Furthermore, we obtain the following.

\noindent $(1).$ $E_2$ and $E_3$ satisfy
\begin{equation}\label{eq-5 : derive G}
\begin{aligned}
|E_2| &\leq CG[|Du|^2]|\nabla Du|^2, \\
|E_3| &\leq CG[|Du|^2]\big(|\nabla Du|^3 +|\nabla^2 Du||\nabla Du| \big). \\
\end{aligned}
\end{equation}

\noindent $(2).$ For $1\leq \alpha \leq l$
\[
\| \partial^{\alpha}G[|Du|^2]~ \partial^{l-\alpha}
Du\|_{L^2}+\|E_\alpha
\partial^{l-\alpha}Du \|_{L^2}
\]
\begin{equation}\label{eq-3 : derive G}
\leq C \|G[ |Du|^2 ]\|_{L^\infty} \big (\| Du\|_{L^\infty} + \|
Du\|_{L^\infty}^\alpha \big)\|\nabla^l Du\|_{L^2}.
\end{equation}

\noindent $(3).$ In case that $l\geq 4$, there exists $\beta$ with
$0<\beta \leq l$ such that the following is satisfied:
\begin{itemize}
\item[(i)] If $\alpha\leq l-1$, then
\begin{equation}\label{eq-4 : derive G}
\begin{aligned}
\| \partial^{\alpha}G[|Du|^2]~ \partial^{l-\alpha} Du\|_{L^2} \leq C \|G[ |Du|^2 ]\|_{L^\infty} \| Du\|_{L^\infty}^{\beta}\|\nabla^{l-1}
Du\|_{L^2}^{\frac{2l-3}{2l-5}}.
\end{aligned}
\end{equation}
\item[(ii)] If $\alpha\leq l$, then
\begin{equation}\label{eq-4 : derive G-v1}
\begin{aligned}
\|E_\alpha
\partial^{l-\alpha}Du \|_{L^2} \leq C \|G[ |Du|^2 ]\|_{L^\infty} \| Du\|_{L^\infty}^{\beta}\|\nabla^{l-1}
Du\|_{L^2}^{\frac{2l-3}{2l-5}}.
\end{aligned}
\end{equation}
\end{itemize}
\end{lemma}

The proof of Lemma \ref{deriv-G} will be given at the Appendix,
since it is a bit lengthy.

\smallskip
Next, we estimate the difference of the viscous part of the stress
tensor, which is useful for uniqueness of regular solutions.
Although it seems elementary, we give the details for clarity.

\begin{lemma}\label{r-400}
Let $v,w \in W^{1,2}(\bbr^3)$. Under the assumptions on $G$ given in
\eqref{G : property}, we have
\begin{equation*}\label{eq 2 : r-400}
m_0\| Dv-Dw\|^2_{L^2(\bbr^3)}\leq
\int_{\bbr^3}\Big(G[|Dv|^2]Dv-G[|Dw|^2]Dw\Big): (Dv-Dw)\,dx,
\end{equation*}
where $m_0$ is a positive constant in \eqref{G : property}.
\end{lemma}

\begin{proof}
We note that
\[
\int_{\R^3}\Big(G[|D v|^2]D v -G[|D w|^2]D w\Big): (D v-D w).
\]
\[
=\int_{\R^3} \Big(\int_0^1\frac{d}{d\theta}\Big(G \big[\big|\theta D
v +(1-\theta)D w\big|^2 \big]\big(\theta D v +(1-\theta)D
w\big)\Big)\,d\theta\Big) :(Dv-Dw)
\]
\[
=2\int_{\R^3}\Big ( \int_0^1 G^{\prime}\big[\big|\theta D v
+(1-\theta)D w\big|^2 \big]\big(\theta D v +(1-\theta)D w : Dv-Dw
\big) \big(\theta D v +(1-\theta)D w \big)\,d\theta \Big ):(D v-D w)
\]
\[
+\int_{\R^3} \Big(\int_0^1G \big[\big|\theta D v +(1-\theta)D
w\big|^2 \big](D v-D w)\,d\theta \Big):(D v-D w)
\]
\[
=2\int_{\R^3} \int_0^1 G^{\prime}\big[\big|\theta D v +(1-\theta)D
w\big|^2 \big]\big(\theta D v +(1-\theta)D w : Dv-Dw
\big)^2\,d\theta
\]
\[
+\int_{\R^3}\int_0^1 G \big[\big|\theta D v +(1-\theta)D
w\big|^2]\,d\theta \big|D v-D w|^2.
\]
\begin{equation}\label{eq-4 : G}
\geq m_0\int_{\R^3}|D v-D w|^2\,d\theta \,dx.
\end{equation}
Due to the properties in \eqref{G : property} for $G$, namely
$G[s]\geq m_0$ and $G[s] + 2G^{\prime}[s] s\geq m_0$ ~for any
$s\in[0,\infty) $, we deduce the inequality \eqref{eq-4 : G}.
Indeed, for any $3\times 3$ matrices $A$ and $B$, we have
\begin{equation}\label{eq-1 : G}
G[|A|^2]|B|^2 + 2 G^{\prime}[|A|^2](A:B)^2 \geq m_0 |B|^2.
\end{equation}
Since, if $G^{\prime}[|A|^2] \geq 0$ then
\begin{equation}\label{eq-2 : G}
G[|A|^2]|B|^2 + 2 G^{\prime}[|A|^2](A:B)^2 \geq  G[|A|^2]|B|^2 \geq m_0|B|^2.
\end{equation}
In case that $G^{\prime}[|A|^2] <0$, we note that
\begin{equation}\label{eq-3 : G}
G[|A|^2]|B|^2 + 2 G^{\prime}[|A|^2](A:B)^2 \ge \big(G[|A|^2] + 2
G^{\prime}[|A|^2]|A|^2  \big)|B|^2\geq m_0 |B|^2.
\end{equation}
We combine \eqref{eq-2 : G} and \eqref{eq-3 : G} to conclude
\eqref{eq-1 : G}. We exploit \eqref{eq-1 : G} with $A=\theta D v
+(1-\theta)D w$ and $B=D v-D w$ to get \eqref{eq-4 : G}. This
completes the proof.
\end{proof}

\section{Proof of Theorems \ref{main}}

In this section, we prove the existence of a local solution to the
equation \eqref{NNS}--\eqref{ini}. We first obtain a priori
estimates and we then justify the estimates by using Galerkin
method.

\subsection{A priori estimate}

We suppose that $u$ is regular. We then compute certain a priori
estimates.

\noindent$\bullet$ ($\|u \|_{L^2}$-estimate)\, We  multiply $u$ to
\eqref{NNS} and integrate it by parts to get
\begin{equation}\label{eq-1 : H0}
\frac{1}{2}\frac{d}{dt}||u||^2_{L^2(\R^3)}+\int_{\R^3}G[|Du|^2]
|Du|^2\,dx=0.
\end{equation}
\noindent$\bullet$ ($\|\nabla u \|_{L^2}$-estimate)\, Taking
derivative $\partial_{x_i}$ to \eqref{NNS} and multiplying
$\partial_{x_i}u$,
\begin{equation*}\label{eq-1 : H1}
\frac{1}{2}\frac{d}{dt}||\partial_{x_i}u||^2_{L^2(\R^3)}+\int_{\R^3}\partial_{x_i}
(G[|Du|^2] Du): \partial_{x_i} Du\,dx =- \int_{\R^3}\partial_{x_i}
\big((u\cdot \nabla) u \big )\cdot \partial_{x_i} u\,dx.
\end{equation*}
Noting that
\begin{equation*}\label{eq-2 : H1}
\begin{aligned}
\partial_{x_i}(G[|Du|^2] Du): \partial_{x_i} Du&=\big[ \partial_{x_i}G[|Du|^2] Du + G[|Du|^2]\partial_{x_i} Du \big] :  \partial_{x_i} Du\\
&= 2 G^{'}[|Du|^2](D u : \partial_{x_i}Du)( Du: \partial_{x_i} Du) + G[|Du|^2]|\partial_{x_i} Du|^2\\
&= 2 G^{'}[|Du|^2]| Du: \partial_{x_i} Du|^2 +
G[|Du|^2]|\partial_{x_i} Du|^2,
\end{aligned}
\end{equation*}
we have
\begin{equation}\label{eq-4 : H1}
\frac{1}{2}\frac{d}{dt}||\partial_{x_i}u||^2_{L^2(\R^3)}+\int_{\R^3}G[|Du|^2]|\partial_{x_i}
Du|^2\,dx + \int_{\R^3}2G^{'}[|Du|^2]| Du: \partial_{x_i} Du|^2\,dx
\end{equation}
\[
= -\int_{\R^3}\partial_{x_i} \big((u\cdot \nabla) u \big )\cdot
\partial_{x_i} u\,dx.
\]
Using $A=Du$ and $B=\partial_{x_i} D u$, we apply the inequality
\eqref{eq-4 : G} to \eqref{eq-4 : H1}, and get
\begin{equation}\label{eq-3 : H1}
\frac{1}{2}\frac{d}{dt}||\partial_{x_i}u||^2_{L^2(\R^3)}+\int_{\R^3}m_0 |\partial_{x_i} Du|^2\,dx
\leq -\int_{\R^3}\partial_{x_i} \big((u\cdot \nabla) u \big )\cdot
\partial_{x_i} u\,dx.
\end{equation}
We will treat the term in righthand side caused by convection
together later.

\noindent$\bullet$ ($\|\nabla^2 u \|_{L^2}$-estimate)\, Taking the
derivative $\partial_{x_{j}}\partial_{x_{i}}$ on \eqref{NNS} and
multiplying it by $\partial_{x_{j}}\partial_{x_{i}}u $,
\begin{equation}\label{eq-2 : H2}
\frac{1}{2}\frac{d}{dt}||\partial_{x_{j}}\partial_{x_{i}}u||^2_{L^2(\R^3)}+\int_{\R^3}\partial_{x_{j}}\partial_{x_{i}}\Big[G[|Du|^2]
D u\Big]: \partial_{x_{j}}\partial_{x_{i}}D u\,dx
\end{equation}
\[
= -\int_{\R^3}\partial_{x_{j}}\partial_{x_{i}}\big((u\cdot\nabla u)
\big)\cdot \partial_{x_{j}}\partial_{x_{i}} u\,dx.
\]
We observe that
\begin{equation}\label{eq-3 : H2}
\begin{aligned}
&\int_{\R^3}\partial_{x_{j}}\partial_{x_{i}}\Big[G[|Du|^2] D u\Big]:\partial_{x_{j}}\partial_{x_{i}}D u\,dx\\
&=\int_{\R^3}G[|Du|^2]|\partial_{x_{j}}\partial_{x_{i}}D u|^2\,
+\sum_{\sigma}\int_{\R^3}\partial_{x_{\sigma(i)}}G[|Du|^2] (\partial_{x_{\sigma(j)}}D u : \partial_{x_{j}}\partial_{x_{i}}D u)\,dx\\
&\hspace{2cm}+\int_{\R^3}\partial_{x_{j}}\partial_{x_{i}}G[|Du|^2](D
u :
\partial_{x_{j}}\partial_{x_{i}}D u)\,dx=:I_{21}+I_{22}+I_{23},
\end{aligned}
\end{equation}
where $\sigma :\{ i,j\}\rightarrow \{i,j \}$ is a permutation of
$\{i,j \}$. We separately estimate terms $I_{22}$ and $I_{23}$ in
\eqref{eq-3 : H2}. Using H\"{o}lder, Young's and Gagliardo-Nirenberg
inequalities, we have for $I_{22}$
\[
|I_{22}|=\abs{\int_{\R^3}  2G^{'}[|Du|^2](Du:
\partial_{x_{\sigma(i)}}D u)(\partial_{x_{\sigma(j)}}D u :
\partial_{x_{j}}\partial_{x_{i}}Du)\,dx}
\]
\[
\leq  C\|G[|Du|^2]\|_{L^\infty}\|\nabla Du\|^2_{L^4}\|\nabla^2
Du\|_{L^2}
\]
\begin{equation*}\label{eq-5 : H2}
\leq   C\|G[|Du|^2]\|_{L^\infty}\|Du\|_{L^\infty}\|\nabla^2
Du\|^2_{L^2},
\end{equation*}
where we used the condition \eqref{G : property}.

\noindent  For $I_{23}$, using Lemma \ref{deriv-G}, we compute
\begin{equation*}
\begin{aligned}
I_{23}&= \int_{\R^3}2\big( G^{'}[|Du|^2](Du : \partial_{x_{j}}\partial_{x_{i}}Du) + E_2 \big) (D u : \partial_{x_{j}}\partial_{x_{i}}Du)\,dx\\
& =\int_{\R^3}E_2(D u : \partial_{x_{j}}\partial_{x_{i}}Du)\,dx
+ 2\int_{\R^3}G^{'}[|Du|^2]  |Du: \partial_{x_{j}}\partial_{x_{i}}Du|^2\,dx\\
&:=I_{231}+I_{232}.
\end{aligned}
\end{equation*}
The term $I_{231}$ is estimated as
\begin{equation}\label{eq-8 : H2}
\begin{aligned}
\abs{ I_{231}}&\leq C\|G[|Du|^2]\|_{L^\infty} \|Du\|_{L^\infty}\|\nabla Du\|^2_{L^4}\|\nabla^2 Du\|_{L^2}\\
 &\leq  C\|G[|Du|^2]\|_{L^\infty} \|Du\|_{L^\infty}^2\|\nabla^2
 Du\|_{L^2}^2,
\end{aligned}
\end{equation}%
where we used the first inequality of \eqref{eq-5 : derive G}. We
combine estimates \eqref{eq-2 : H2}-\eqref{eq-8 : H2} to get
\[
\frac{1}{2}\frac{d}{dt}||\partial_{x_{j}}\partial_{x_{i}}u||^2_{L^2(\R^3)}+\int_{\R^3}G[|Du|^2]|\partial_{x_{j}}\partial_{x_{i}}D
u|^2 + \int_{\R^3}2G^{'}[|Du|^2]|Du:
\partial_{x_{j}}\partial_{x_{i}}Du|^2
\]
\begin{equation}\label{eq-4 : H2}
\leq C\|G[|Du|^2]\|_{L^\infty}(\|Du\|_{L^\infty} +
\|Du\|_{L^\infty}^2)\|\nabla^2 Du\|^2_{L^2}
-\int_{\R^3}\partial_{x_{j}}\partial_{x_{i}}\big((u\cdot\nabla u)
\big)\cdot \partial_{x_{j}}\partial_{x_{i}} u.
\end{equation}
Similarly as in \eqref{eq-3 : H1}, we have
\[
\frac{1}{2}\frac{d}{dt}||\partial_{x_{j}}\partial_{x_{i}}u||^2_{L^2(\R^3)}+\int_{\R^3}m_0|\partial_{x_{j}}\partial_{x_{i}}D
u|^2
\]
\begin{equation}\label{eq-1 : H2}
\leq C\|G[|Du|^2]\|_{L^\infty}(\|Du\|_{L^\infty} +
\|Du\|_{L^\infty}^2)\|\nabla^2 Du\|^2_{L^2}
-\int_{\R^3}\partial_{x_{j}}\partial_{x_{i}}\big((u\cdot\nabla u)
\big)\cdot \partial_{x_{j}}\partial_{x_{i}} u.
\end{equation}
\noindent$\bullet$ ($\|\nabla^3 u \|_{L^2}$-estimate) For
convenience, we denote
$\partial^3:=\partial_{x_{k}}\partial_{x_{j}}\partial_{x_{i}}$.
Similarly as before, taking the derivative $\partial^3$ on
\eqref{NNS} and multiplying it by $\partial^3 u$,
\begin{equation}\label{eq-1 : H3}
\frac{1}{2}\frac{d}{dt}||\partial^3u||^2_{L^2(\R^3)}
+\int_{\R^3}\partial^3\Big[G[|Du|^2] D u\Big] :
\partial^3D u\,dx=-\int_{\R^3}\partial^3\big((u\cdot\nabla u) \big)\cdot
\partial^3 u\,dx.
\end{equation}
Direct computations show that
\[
\int_{\R^3}\partial^3\Big[G[|Du|^2] D u\Big] : \partial^3D
u\,dx=\int_{\R^3}G[|Du|^2]|\partial^3D u|^2\,dx
\]
\[
+\sum_{\sigma_3}\int_{\R^3}\partial_{x_{\sigma_3(i)}}G[|Du|^2](\partial_{x_{\sigma_3(k)}}\partial_{x_{\sigma_3(j)}}D
u:\partial^3D u)\,dx
\]
\[
+\sum_{\sigma_3}\int_{\R^3}\partial_{x_{\sigma_3(j)}}\partial_{x_{\sigma_3(i)}}G[|Du|^2](\partial_{x_{\sigma_3(k)}}D
u : \partial^3Du)\,dx
\]
\begin{equation}\label{eq-2 : H3}
+\int_{\R^3}\partial^3G[|Du|^2](D u:\partial^3D
u)\,dx=I_{31}+I_{32}+I_{33}+I_{34},
\end{equation}
where $\sigma_3=\pi_3\circ \tilde{\sigma}_3$ such that
$\tilde{\sigma}_3:\{i,j,k\}\rightarrow \{i,j,k\}$ is a permutation
of $\{i,j,k \}$ and $\pi_3$ is a mapping from $\{i,j,k\}$ to
$\{1,2,3\}$.

We separately estimate terms $I_{32}$, $I_{33}$ and $I_{34}$. We
note first that
\begin{equation}\label{eq-11 : H3}
\begin{aligned}
|I_{32}|&\leq \int_{\R^3}|2(G^{'}[|Du|^2]|
|Du||\partial_{x_{\sigma_3(i)}}Du||\partial_{x_{\sigma_3(k)}}\partial_{x_{\sigma_3(j)}}
D u|
|\partial^3 D u|\,dx\\
&\leq C\|G[|Du|^2]\|_{L^\infty}\|\nabla Du\|_{L^6}\|\nabla^2 Du\|_{L^3}\|\nabla^3 Du\|_{L^2}\\
&\leq C\|G[|Du|^2]\|_{L^\infty}\|\nabla^2 Du\|_{L^2} \| Du\|_{L^\infty}^{\frac{1}{3}} \|\nabla^3 Du\|^{\frac{2}{3}}_{L^2}\|\nabla^3Du\|_{L^2}\\
&\leq C\|G[|Du|^2]\|^6_{L^\infty}\| Du\|_{L^\infty}^2 \|\nabla^2
Du\|^{6}_{L^2}+\epsilon\|\nabla^3 Du\|^{2}_{L^2}.
\end{aligned}
\end{equation}
For $I_{33}$, we have
\[
|I_{33}|=\abs{\int_{\R^3}(
2G^{'}[|Du|^2]Du:\partial_{x_{\sigma_3(j)}}
\partial_{x_{\sigma_3(i)}} D u
+ E_2  )(\partial_{x_{\sigma_3(k)}}D u : \partial^3Du)\,dx}
\]
\[
\leq \int_{\R^3}2(G^{'}[|Du|^2]|Du||\nabla^2 D u| + G[|Du|^2]|\nabla
D u|^2  )|\nabla D u| |\nabla^3 Du|\,dx
\]
\[
\leq  C \|G^{'}[|Du|^2]Du\|_{L^\infty} \|\nabla^2 Du\|_{L^3}\|\nabla
Du\|_{L^6}\|\nabla^3 Du\|_{L^2} +\|G[|Du|^2]\|_{L^\infty}\|\nabla
Du\|^3_{L^6}\|\nabla^3 Du\|_{L^2}
\]
\[
\leq C\|G[|Du|^2]\|^6_{L^\infty}\|Du \|_{L^\infty}^2\|\nabla^2
Du\|^{6}_{L^2}
+C\|G[|Du|^2]\|^2_{L^\infty}\|\nabla^2Du\|^6_{L^2}+2\epsilon\|\nabla^3
Du\|^{2}_{L^2}
\]
\begin{equation}\label{eq-8 : H3}
\leq C(\|G[|Du|^2]\|^6_{L^\infty}\|Du
\|_{L^\infty}^2+\|G(|Du|)\|^2_{L^\infty})\|\nabla^2
Du\|^{6}_{L^2}+2\epsilon\|\nabla^3 Du\|^{2}_{L^2},
\end{equation}
where we use same argument as \eqref{eq-11 : H3} in the fourth
inequality. Finally, for $I_{34}$, using Lemma \ref{deriv-G}, we
note that
\begin{equation}\label{eq-5 : H3}
\begin{aligned}
I_{34}=&\int_{\R^3}\big( 2G^{'}[|Du|^2]Du:\partial^3 D u) + E_3\big)(D u:\partial^3D u)\,dx\\
=&2\int_{\R^3} G^{'}[|Du|^2]|Du:\partial^3 D u|^2\,dx + \int_{\R^3}
E_3(D u:\partial^3D u)\,dx.
\end{aligned}
\end{equation}
The second term in \eqref{eq-5 : H3} is estimated as follows:
\[
\int_{\R^3}E_3(D u:\partial^3D u)\,dx \leq  \int_{\R^3}|E_3||D
u||\nabla^3 D u|\,dx
\]
\[
\leq C\int_{\R^3} G[|Du|^2]\big(|\nabla Du|^3 +|\nabla^2 Du||\nabla
Du| \big)|D u||\nabla^3 D u|\,dx
\]
\[
\leq C\|G[|Du|^2]\|_{L^\infty}\|Du\|_{L^\infty} \big( \|\nabla
Du\|_{L^6}^3 + \|\nabla Du\|_{L^6}\|\nabla^2 Du\|_{L^3} \big
)\|\nabla^3 D u\|_{L^2}
\]
\[
\leq C\|G[|Du|^2]\|^2_{L^\infty}\|Du\|^2_{L^\infty}\|\nabla^2
Du\|^6_{L^2}+C\|G[|Du|^2]\|^6_{L^\infty}\|Du\|^4_{L^\infty}\|\nabla^2
Du\|^{6}_{L^2}+2\epsilon\|\nabla^3 Du\|^{2}_{L^2},
\]
\begin{equation}\label{eq-7 : H3}
\leq
C(\|G[|Du|^2]\|^2_{L^\infty}\|Du\|^2_{L^\infty}+\|G[|Du|^2]\|^6_{L^\infty}\|Du\|^4_{L^\infty})\|\nabla^2
Du\|^{6}_{L^2}+2\epsilon\|\nabla^3 Du\|^{2}_{L^2},
\end{equation}
where we use same argument as \eqref{eq-8 : H3} in the third
inequality. Adding up the estimates \eqref{eq-1 : H3}-\eqref{eq-7 :
H3}, we obtain
\[
\frac{d}{dt}||\partial^3u||^2_{L^2(\R^3)}
+\int_{\R^3}G[|Du|^2]|\partial^3D u|^2\,dx
+\int_{\R^3}2G^{'}[|Du|^2]|Du:
\partial^3Du|^2\,dx
\]
\[
\leq C(\|G[|Du|^2]\|^2_{L^\infty}+  \|G[|Du|^2]\|^6_{L^\infty}
)(\|Du\|^2_{L^\infty}+\|Du\|^4_{L^\infty})\|\nabla^2
Du\|^{6}_{L^2}+5\epsilon\|\nabla^3 Du\|^{2}_{L^2}
\]
\begin{equation}\label{eq-13 : H3}
-\int_{\R^3}\partial^3\big((u\cdot\nabla u) \big)\cdot
\partial^3u\,dx.
\end{equation}
Hence, we have
\[
\frac{d}{dt}||\partial^3u||^2_{L^2(\R^3)}
+\int_{\R^3}m_0|\partial^3D u|^2\,dx
\]
\[
\leq C(\|G[|Du|^2]\|^2_{L^\infty}+  \|G[|Du|^2]\|^6_{L^\infty}
)(\|Du\|^2_{L^\infty}+\|Du\|^4_{L^\infty})\|\nabla^2
Du\|^{6}_{L^2}+5\epsilon\|\nabla^3 Du\|^{2}_{L^2}
\]
\begin{equation}\label{eq-9 : H3}
-\int_{\R^3}\partial^3\big((u\cdot\nabla u) \big)\cdot
\partial^3u\,dx.
\end{equation}
Next, we estimate the terms caused by convection terms in
\eqref{eq-3 : H1}, \eqref{eq-1 : H2} and \eqref{eq-9 : H3}.
\begin{align}\label{s-2-2}
\begin{aligned}
\sum_{1\leq |\alpha|\leq 3}\int_{\R^3} \partial^{\alpha}[(u\cdot
\nabla)u]\cdot \partial^{\alpha}u\,dx
&=\sum_{1\leq |\alpha|\leq 3}\int_{\R^3} [\partial^{\alpha}((u\cdot \nabla)u)-u\cdot \nabla \partial^{\alpha}u]\partial^{\alpha}u\,dx\\
&\leq \sum_{1\leq |\alpha|\leq3}\|\partial^{\alpha}((u\cdot \nabla)u)-u\cdot \nabla \partial^{\alpha}u\|_{L^2}\|\partial^{\alpha} u\|_{L^{2}}\\
&\leq \sum_{1\leq |\alpha|\leq3} \|\nabla u\|_{L^\infty} \|u\|_{H^{3}}\|\partial^{\alpha} u\|_{L^{2}}\\
&\leq C\|\nabla u\|_{L^\infty}\| u\|^2_{H^{3}},
\end{aligned}
\end{align}
where we use the following inequality:
\begin{equation}\label{est-comm}
\sum_{|\alpha|\leq m}\int_{\R^3} \|\nabla^{\alpha}(
fg)-(\nabla^{\alpha}f)g\|_{L^2}\leq C(\|f\|_{H^{m-1}}\|\nabla
g\|_{L^{\infty}}+\|f\|_{L^{\infty}}\|g\|_{H^m}).
\end{equation}
We combine \eqref{eq-1 : H0}, \eqref{eq-3 : H1}, \eqref{eq-1 : H2}
and \eqref{eq-9 : H3} with \eqref{s-2-2} to conclude
\begin{equation}\label{final-est-H3-local}
\begin{aligned}
&\frac{d}{dt}||u||^2_{H^3(\R^3)} +\int_{\R^3}(m_0-5\epsilon) (|\nabla^3 Du|^2+|\nabla^2 Du|^2 + |\nabla Du|^2 + | Du|^2)\,dx\\
& \leq C\|\nabla u\|_{L^\infty}\| u\|^2_{H^{3}}
+C\|G[|Du|^2]\|_{L^\infty}(\|Du\|^2_{L^\infty}+\|Du\|_{L^\infty})\|\nabla^2
Du\|^2_{L^2}\\
&+C(\|G[|Du|^2]\|^2_{L^\infty}+\|G[|Du|^2]\|^6_{L^\infty})(\|Du\|^2_{L^\infty}+\|Du\|^4_{L^\infty})\|\nabla^2
Du\|^{6}_{L^2}.
\end{aligned}
\end{equation}
Here we choose a sufficiently small $\epsilon>0$ such that
$m_0-5\epsilon>0$. Furthermore, we have
\begin{equation}\label{eq-2 : H3 local}
\|G[|Du|^2]\|_{L^\infty} \leq \max_{0\leq s\leq
\|Du|\|_{L^\infty}}G[s] \leq \max_{0\leq s\leq
C\|u|\|_{H^3}}G[s]:=g(\|u\|_{H^3}),
\end{equation}
where $g:[0,\infty) \mapsto [0,\infty)$ is a nondecreasing function.
We set $X(t) := \|u(t)\|_{H^3(\R^3)}$ and it then follows from
\eqref{final-est-H3-local} and \eqref{eq-2 : H3 local} that
\begin{equation}\label{eq-1 : H3 local}
\frac{d}{dt} X^2\leq f_3(X)X^2
\end{equation}
for some nondecreasing continuous function $f_3$, which immediately
implies that there exists $T_3>0$ such that
$$ \sup_{0\leq t\leq T_3}X(t) < \infty.$$

\noindent $\bullet$ ($\|\nabla^4 u \|_{L^2}$-estimate) For
simplicity, we denote
$\partial^4:=\partial_{x_{l}}\partial_{x_{k}}\partial_{x_{j}}\partial_{x_{i}}$.
Similarly as in \eqref{eq-1 : H3} and \eqref{eq-2 : H3}, we have
\begin{equation}\label{eq-1 : H4}
\frac{1}{2}\frac{d}{dt}||\partial^4u||^2_{L^2(\R^3)}
+\int_{\R^3}\partial^4\Big[G[|Du|^2] D u\Big] : \partial^4 D u\,dx
=-\int_{\R^3}\partial^4\big((u\cdot\nabla u) \big)\cdot
\partial^4 u\,dx.
\end{equation}
We note that
\begin{equation}\label{eq-2 : H4}
\begin{aligned}
&\int_{\R^3}\partial^4\Big[G[|Du|^2] D u\Big] :\partial^4 D u\,dx
=\int_{\R^3}G[|Du|^2]|\partial^4D u|^2\,dx\\
&\hspace{0.5cm}+\sum_{\sigma_4}\int_{\R^3}\partial_{x_{\sigma_4(i)}}G[|Du|^2](\partial_{x_{\sigma_4(l)}}\partial_{x_{\sigma_4(k)}}\partial_{x_{\sigma_4(j)}}D
u :
\partial^4D u)\,dx\\
&\hspace{0.5cm}+\sum_{\sigma_4}\int_{\R^3}\partial_{x_{\sigma_4(j)}}\partial_{x_{\sigma_4(i)}}G[|Du|^2](\partial_{x_{\sigma_4(l)}}\partial_{x_{\sigma_4(k)}}D
u
:\partial^4Du)\,dx\\
&\hspace{0.5cm}+\sum_{\sigma_4}\int_{\R^3}\partial_{x_{\sigma_4(k)}}\partial_{x_{\sigma_4(j)}}\partial_{x_{\sigma_4(i)}}G[|Du|^2](\partial_{x_{l}}D
u
:\partial^4D u)\,dx\\
&\hspace{0.5cm}+\int_{\R^3}\partial^4G[|Du|^2] (D u:\partial^4D
u)\,dx:= I_{41}+I_{42}+I_{43}+I_{44}+I_{45},
\end{aligned}
\end{equation}
where $\sigma_4=\pi_4\circ \tilde{\sigma}_4$ such that
$\tilde{\sigma}_4:\{i,j,k,l\}\rightarrow \{i,j,k,l\}$ is a
permutation of $\{i,j,k,l \}$ and $\pi_4$ is a mapping from
$\{i,j,k,l\}$ to $\{1,2,3\}$.

We first estimate $I_{42}$. Exploiting \eqref{eq-4 : derive G} with
$l=4,~\alpha=1$, we get
\begin{equation*}\label{eq-10 : H4}
\begin{aligned}
|I_{42}|&\leq C \| \partial G[|Du|^2] ~ \partial^3 Du\|_{L^2} \| \partial^4 Du\|_{L^2}\\
&\leq C \|G[ |Du|^2 ]\|_{L^\infty} \|Du \|_{L^\infty}^{\beta_1} \|\partial^3 Du \|_{L^2}^{5/3}\| \partial^4 Du\|_{L^2}\\
&\leq C \|G[ |Du|^2 ]\|_{L^\infty}^2 \|Du \|_{L^\infty}^{2\beta_1}
\|\partial^3 Du \|_{L^2}^{10/3}+ \epsilon\|\nabla^4 Du\|_{L^2}^2
\end{aligned}
\end{equation*}
for some $0<\beta_1\leq 4$. Similarly, using \eqref{eq-4 : derive G}
with $l=4,~\alpha=2$, we have
\begin{equation*}
\begin{aligned}
|I_{43}|&\leq C \| \partial^2 G[|Du|^2] ~ \partial^2 Du\|_{L^2} \| \partial^4 Du\|_{L^2}\\
&\leq C \|G[ |Du|^2 ]\|_{L^\infty} \|Du \|_{L^\infty}^{\beta_2} \|\partial^3 Du \|_{L^2}^{5/3}\| \partial^4 Du\|_{L^2}\\
&\leq C \|G[ |Du|^2 ]\|_{L^\infty}^2 \|Du \|_{L^\infty}^{2\beta_2}
\|\partial^3 Du \|_{L^2}^{10/3}+ \epsilon\|\nabla^4 Du\|_{L^2}^2
\end{aligned}
\end{equation*}
for some $0<\beta_2\leq 4$. Again, due \eqref{eq-4 : derive G} with
$l=4,~\alpha=3$, we obtain for some $0<\beta_3\leq 4$
\begin{equation*}
\begin{aligned}
|I_{44}|&\leq C \| \partial^3 G[|Du|^2] ~ \partial Du\|_{L^2} \| \partial^4 Du\|_{L^2}\\
&\leq C \|G[ |Du|^2 ]\|_{L^\infty} \|Du \|_{L^\infty}^{\beta_3} \|\partial^3 Du \|_{L^2}^{5/3}\| \partial^4 Du\|_{L^2}\\
&\leq C \|G[ |Du|^2 ]\|_{L^\infty}^2 \|Du \|_{L^\infty}^{2\beta_3}
\|\partial^3 Du \|_{L^2}^{10/3}+ \epsilon\|\nabla^4 Du\|_{L^2}^2.
\end{aligned}
\end{equation*}
For the term $I_{45}$, we note that
\begin{equation}\label{eq-7 : H4}
\begin{aligned}
I_{45}
&= \int_{\R^3}\big[2G^{'}[|Du|^2] (Du:
\partial^4D u )+ E_4\big]
(D u :\partial^4D u)\,dx\\
&= \int_{\R^3}2G^{'}[|Du|^2] |Du: \partial^4D u|^2\,dx
+\int_{\R^3}E_4 (D u :\partial^4D u)\,dx.
\end{aligned}
\end{equation}
Owing to \eqref{eq-4 : derive G-v1}, we see that
\begin{equation}\label{eq-8 : H4}
\begin{aligned}
\int_{\R^3}E_4(D u :\partial^4D u)\,dx
&\leq C \|G[ |Du|^2 ]\|_{L^\infty} \|Du \|_{L^\infty}^{\beta_4} \|\partial^3 Du \|_{L^2}^{5/3}\| \partial^4 Du\|_{L^2}\\
&\leq C \|G[ |Du|^2 ]\|_{L^\infty}^2 \|Du \|_{L^\infty}^{2\beta_4}
\|\partial^3 Du \|_{L^2}^{10/3}+ \epsilon\|\nabla^4 Du\|_{L^2}^2
\end{aligned}
\end{equation}
for some $0<\beta_4\leq 4$. We combine \eqref{eq-1 : H4}-\eqref{eq-8 : H4} to have
\[
\frac{1}{2}\frac{d}{dt}||\partial^4u||^2_{L^2(\R^3)}
+\int_{\R^3}G[|Du|^2]|\partial^4D u|^2\,dx+
\int_{\R^3}2G^{'}[|Du|^2]|Du:\partial^4Du|^2\,dx
\]
\[
\leq C \|G[ |Du|^2 ]\|_{L^\infty}^2 (1+\|Du
\|_{L^\infty})^{2\beta}\|\nabla^3 Du \|_{L^2}^{10/3}
\]
\begin{equation}\label{eq-11 : H4}
+ 4\epsilon \|\nabla^4
Du\|^2_{L^2}-\int_{\R^3}\partial^4\big((u\cdot\nabla u) \big) \cdot
\partial^4 u\,dx,
\end{equation}
where $\beta=\beta_1+\beta_2+\beta_3+\beta_4$.
Hence, as before, we
have
\[
\frac{1}{2}\frac{d}{dt}||\partial^4u||^2_{L^2(\R^3)}
+\int_{\R^3}m_0|\partial^4D u|^2\,dx
\]
\[
\leq C \|G[ |Du|^2 ]\|_{L^\infty}^2 (1+\|Du
\|_{L^\infty})^{2\beta}\|\nabla^3 Du \|_{L^2}^{10/3}
\]
\begin{equation}\label{eq-9 : H4}
 + 4\epsilon
\|\nabla^4 Du\|^2_{L^2}-\int_{\R^3}\partial^4\big((u\cdot\nabla u)
\big) \cdot
\partial^4 u\,dx.
\end{equation}
Using \eqref{est-comm}, we estimate the convection term
\begin{align}\label{s-3}
\begin{aligned}
\sum_{|\alpha|= 4}\int_{\bbr^3} \partial^{\alpha}[(u\cdot
\nabla)u]\cdot \partial^{\alpha}u\,dx
&=\sum_{|\alpha|= 4}\int_{\R^3} [\partial^{\alpha}((u\cdot \nabla)u)-u\cdot \nabla \partial^{\alpha}u]\partial^{\alpha}u\,dx\\
&\leq \sum_{|\alpha|= 4}\|\partial^{\alpha}((u\cdot \nabla)u)-u\cdot \nabla \partial^{\alpha}u\|_{L^2}\|\partial^{\alpha} u\|_{L^{2}}\\
&\leq C \sum_{|\alpha|= 4} \|\nabla u\|_{L^\infty} \|u\|_{H^{4}}\|\partial^{\alpha} u\|_{L^{2}}\\
&\leq C \|\nabla u\|_{L^\infty} \|u\|^2_{H^{4}}.
\end{aligned}
\end{align}
Finally, we combine \eqref{eq-1 : H0}, \eqref{eq-3 : H1},
\eqref{eq-1 : H2}, \eqref{eq-9 : H3} and \eqref{eq-9 : H4} with
\eqref{s-3} to conclude
\begin{equation*}
\begin{aligned}
&\frac{1}{2}\frac{d}{dt}||u||^2_{H^4(\R^3)} +\int_{\R^3} \big (m_0-9\epsilon\big)(|\nabla^4 Du|^2+|\nabla^3 Du|^2+|\nabla^2 Du|^2 + |\nabla Du|^2
+ | Du|^2)\,dx\\
& \leq C\big( \|G[ |Du|^2 ]\|_{L^\infty}^2+  \|G[ |Du|^2
]\|_{L^\infty}^6\big( \|D u\|_{L^\infty}+\|D u\|^2_{L^\infty}+\|D
u\|^4_{L^\infty} + (1+\|D u\|_{L^\infty})^{2\beta}\big)
\end{aligned}
\end{equation*}
\[
\times \big (\|\nabla^3 Du\|^{10/3}_{L^2} + \|\nabla^2 Du\|^{6}_{L^2} + \|\nabla^2 Du\|^{2}_{L^2}+ \|\nabla Du\|_{L^2}^2\big )\\
+ C\|\nabla u\|_{L^\infty}\| u\|_{H^{4}}^2.
\]
Let us denote $X(t) := \|u(t)\|_{H^4(\R^3)}$. Similarly as in
\eqref{eq-1 : H3 local}, we have
\[
\frac{d}{dt} X^2\leq f_4(X)X^2.
\]
for some non-decreasing continuous function $f_4$. Therefore, there
exists $T_4>0$ such that
\[
\sup_{0\leq t\leq T_4}X(t) < \infty.
\]
So far, we have proven for some non-decreasing continuous function
$f_k$, $k=3,4$
\begin{equation}\label{final-est-Hl-local}
\frac{1}{2}\frac{d}{dt}\| u\|^2_{H^k} +\int_{\R^3}\big
(m_0-\epsilon\big)(|\nabla^k Du|^2+\cdots+ |\nabla Du|^2 + |
Du|^2)\,dx \leq f_k(\| u\|_{H^k}) \| u\|^2_{H^k}.
\end{equation}

Next, we will show that \eqref{final-est-Hl-local} holds for general
$k\geq 3$ by the induction argument. Suppose
\eqref{final-est-Hl-local} is true for $k=l-1$ for some $l\geq 4$.
We then prove that \eqref{final-est-Hl-local} is true for $k=l$.

Indeed, let $\sigma_l=\pi_l\circ \tilde{\sigma}_l$ such that
$\tilde{\sigma}_l:\{1,2,\cdots,l\}\rightarrow \{1,2,\cdots,l\}$ is a
permutation of $\{1,2,\cdots,l \}$ and $\pi_l$ is a mapping from
$\{1,2,\cdots,l\}$ to $\{1,2,3\}$. For simplicity, we denote
$\partial^l:=\partial_{x_{\sigma(l)}}\partial_{x_{\sigma(l-1)}}\cdots\partial_{x_{\sigma(1)}}$.
Similar computations as before leads to
\[
\frac{1}{2}\frac{d}{dt}||\partial^l u||^2_{L^2(\R^3)}
+\int_{\R^3}G[|Du|^2]|\partial^l D u|^2\,dx +\int
2G^{\prime}[|Du|^2]|Du:\partial^l Du|^2\,dx
\]
\[
\leq \sum_{|\alpha|=1}^{l-1}\int_{\R^3}|\partial^\alpha G[|Du|^2]|
|\partial^{l-\alpha} D u| |\partial^lD u| \,dx
 + \int_{\R^3}|E_{l}||D u||\partial^l D u|\,dx
\]
\begin{equation}\label{eq-6 : I-l1}
-\int_{\R^3}\partial^l\big(u\cdot\nabla u\big) \cdot\partial^l
u\,dx:=  I_{l1}+ I_{l2} + I_{l3}.
\end{equation}
We exploit \eqref{eq-4 : derive G} for $ I_{l1}$ to get
\begin{equation}\label{eq-4 : I-l1}
\begin{aligned}
| I_{l1}| &\leq \sum_{\alpha=1}^{l-1}\|\partial^\alpha G[|Du|^2]~\nabla^{l-\alpha} D u\|_{L^2} \|\partial^l D u\|_{L^2} \\
&\leq C \|G[ |Du|^2
]\|_{L^\infty}\sum_{\alpha=1}^{l-1}\|Du\|_{L^\infty}^{\beta_5}
\|\nabla^{l-1} D u\|_{L^2}^{\frac{2l-3}{2l-5}}
\|\nabla^l D u\|_{L^2}\\
&\leq f_{l1}(\| u\|_{H^{l-1}})\|
u\|_{H^{l}}^{\frac{2l-3}{2l-5}}\|\nabla^l D u\|_{L^2}
\end{aligned}
\end{equation}
for some function $f_{l1}$. For $I_{l2}$, due to \eqref{eq-4 :
derive G-v1}, we obtain
\begin{equation}\label{eq-5 : I-l1}
\begin{aligned}
|I_{l2}|&\leq \|E_l D u\|_{L^2} \|\partial^l D u\|_{L^2} \\
&\leq C \|G[ |Du|^2 ]\|_{L^\infty} \|Du\|_{L^\infty}^{\beta_6} \|\nabla^{l-1} D u\|_{L^2}^{\frac{2l-3}{2l-5}} \|\nabla^l D u\|_{L^2}\\
&\leq f_{l2}(\| u\|_{H^{l-1}})\|
u\|_{H^{l}}^{\frac{2l-3}{2l-5}}\|\nabla^l D u\|_{L^2}
\end{aligned}
\end{equation}
for some function $f_{l2}$. Lastly, we estimate $I_{l3}$ similarly
as before.
\begin{align}\label{s-2-higher}
\begin{aligned}
&\sum_{1\leq |\alpha|\leq l}\int_{\R^3} \partial^{\alpha}[(u\cdot
\nabla)u]\cdot \partial^{\alpha}u\,dx
=\sum_{1\leq |\alpha|\leq l}\int_{\R^3} [\partial^{\alpha}((u\cdot \nabla)u)-u\cdot \nabla \partial^{\alpha}u]\partial^{\alpha}u\,dx\\
&\qquad\leq \sum_{1\leq |\alpha|\leq l} \|\nabla u\|_{L^\infty}
\|u\|_{H^l}\|\partial^{\alpha} u\|_{L^{2}} \leq C\|\nabla
u\|_{L^\infty} \|u\|^2_{H^{l}}.
\end{aligned}
\end{align}
Combining \eqref{eq-6 : I-l1}, \eqref{eq-4 : I-l1}, \eqref{eq-5 :
I-l1} and \eqref{s-2-higher}, we obtain
\[
\frac{1}{2}\frac{d}{dt}||\nabla^l u||^2_{L^2(\R^3)} +\int_{\R^3} m_0
|\nabla^l D u|^2\,dx
\]
\[
\leq f_l(\|u\|_{H^{l-1}})
\|u\|_{H^{l}}^{\frac{2l-3}{2l-5}}\|\nabla^l Du\|_{L^2} + C\|\nabla
u\|^2_{L^\infty}\|u\|^2_{H^{l}}
\]
\[
\leq f_l(\|u\|_{H^{l-1}}) \|u\|_{H^{l}}^{^{\frac{2(2l-3)}{2l-5}}} +
C\|\nabla u\|^2_{L^\infty}\|u\|^2_{H^{l}}+ \epsilon \|\nabla^l Du
\|_{L^2}^2
\]
for some nondecreasing continuous function $f_l$. Hence, we have
\begin{equation*}\label{eq-1 : Hl local}
\frac{1}{2}\frac{d}{dt}||\nabla^l u||^2_{L^2(\R^3)}+\int_{\R^3}\big(
m_0-\epsilon \big)|\nabla^l D u|^2\,dx\leq f_l(\|u \|_{H^{l}})\|
u\|^2_{H^{l}}.
\end{equation*}
Since \eqref{final-est-Hl-local} is true for $k=l-1$, we conclude
that
\begin{equation*}
\frac{1}{2}\frac{d}{dt}|| u||^2_{H^l}+\int_{\R^3}\big(m_0-\epsilon
\big)\big(|\nabla^l D u|^2+ \cdots |\nabla D u|^2 + | D
u|^2\big)\leq f_l(\|u \|_{H^{l}}) \|u\|^2_{H^{l}}.
\end{equation*}
Choosing $\epsilon>0$ so small and letting $X(t):=\|u(t) \|_{H^l}$,
we obtain
\begin{equation*}
\frac{d}{dt}X^2  \leq  f_l(X) X^2,
\end{equation*}
which yields that there exists $T_l>0$ such that
$$ \sup_{0\leq t\leq T_l}\|u(t) \|_{H^l} <\infty. $$
We complete the a priori estimates. \qed

\begin{remark}
We note that $\partial_t{u}\in L^2((0,T_l);L^2(\R^3))$. Indeed, we
introduce the antiderivative of $G$, denoted by $\tilde{G}$, i.e.,
$\tilde{G}[s]=\int^s_0 G[\tau]d\tau$. Multiplying $\partial_t u$ to
\eqref{NNS}, integrating it by parts, and using H\"{o}lder and
Young's inequalities, we have
\begin{equation}\label{reg-time-20}
\frac{1}{2}\int_{\R^3}|\partial_t{u}|^2\,dx
+\frac{1}{2}\frac{d}{dt}\int_{\R^3}\tilde{G}[|Du|^2]\,dx\leq
C\int_{\R^3}|u|^2|\nabla u|^2.
\end{equation}
Again, integrating the estimate \eqref{reg-time-20} over the time
interval $[0,T_l]$, we obtain
\[
\int_0^{T_l}\int_{\R^3}|\partial_t
u|^2\,dxdt+\int_{\R^3}\tilde{G}[|Du(\cdot, T_l)|^2]\,dx
\]
\begin{equation}\label{reg-time-300}
\leq \int_{\R^3}\tilde{G}[|Du_0|^2]\,dx+
C\int_0^{T_l}\int_{\R^3}|u|^2|\nabla u|^2\,dxdt.
\end{equation}
Using Sobolev embedding, the second term in \eqref{reg-time-300} is
estimated as follows:
\[
\int_0^{T_l}\int_{\R^3}|u|^2|\nabla u|^2\,dxdt \leq \int_0^{T_l}
\|u\|^2_{L^{\infty}}\|\nabla u\|^2_{L^2}dt
\]
\begin{equation*}\label{convec-2}
\leq C\sup_{0<\tau\leq T_l}\|u(\tau)\|^2_{W^{2,2}}\int_0^{T_l} \|
\nabla u\|^2_{L^{2}}dt<C.
\end{equation*}
Therefore, we obtain $\partial_t{u} \in L^{2}(0,T; L^2(\R^3))$.\qed
\end{remark}

Using the method of Galerkin approximation, we construct the regular
solution satisfying the a priori estimates above.
\begin{pfthm1}
To precisely justify existence of a regular solution, we proceed by
a Galerkin method. In view of \cite[Lemma 3.10]{P}, there exists a
countable subset $\{w_i \}_{i=1}^\infty$ of the space
$\mathcal{V}:=\{\varphi \in \mathcal{D}(\bbr^3)^3: \nabla \cdot
\varphi=0\}$ that is dense in $H^{l}(\bbr^3)\cap \mathcal{V}$. We
consider Galerkin approximate equation for $u_m(t)=\sum_{i=1}^m
g_i^m(t)w_i$,
\begin{align*}
\begin{aligned}
\int_{\bbr^3}\bke{\partial_t u_m(t)\cdot w +G[|Du_m(t)|^2]
Du_m(t)):  Dw+ (u_m \otimes u_m)(t) : \nabla w }dx=0,\\
\end{aligned}
\end{align*}
where $w\in \mbox{span}\{w_1,w_2,\cdots,w_m\}$ and
$||u_m(0)-u_0||_{H^{l}} \rightarrow 0$ as $m\rightarrow \infty$.
From a priori estimate above, we obtain
\[
\|u_m\|_{L^{\infty}(0,T_l;H^{l}(\bbr^3))}\leq C,
\]
\[
\|u_m\|_{L^2(0,T_l; H^{l+1}(\bbr^3))} \leq C,
\]
\[
\|\partial_t u_m\|_{L^{2}(0,T_l;L^2(\bbr^3))}\leq C,
\]
where $C$ is independent on $m$. Due to the uniform boundedness
above, we can choose a subsequence $u_{m_k}$ of $u_{m}$ such that
\[
u_{m_k} \rightarrow u \quad \mbox{weakly-* in}\quad
L^{\infty}(0,T_1;H^{l}(\bbr^3)).
\]
\[
u_{m_k} \rightarrow u \quad \mbox{weakly in}\quad  L^2(0,T_l;
H^{l+1}(\bbr^3)).
\]
\[
\partial_t u_{m_k} \rightarrow \partial_t u \quad \mbox{weakly in}\quad L^{2}(0,T_l;L^2(\bbr^3)),
\]
as $k\rightarrow\infty$. Using the Aubin-Lions Lemma, we obtain
\begin{equation}\label{NNS-strong}
u_{m_k} \rightarrow u \quad \mbox{strongly in}\quad
L^{2}(0,T_l;H^{l}_{loc}(\bbr^3)), \quad k\rightarrow\infty.
\end{equation}
By the standard argument (see e.g., \cite{M-N-R-R}), we see that
$u\in L^{\infty}(0, T_l;H^{l}(\bbr^3))\cap
L^{2}(0,T_l;H^{l+1}(\bbr^3))$ is a solution of the following
equation in a weak sense
\begin{align*}
\begin{aligned}
\partial_t u  -\nabla \cdot \tilde{G}+ (u\cdot \nabla)u +\nabla p
 =0,
\end{aligned}
\end{align*}
that is, for any $\psi\in C_c^1(\bbr^3\times [0,T])$ such that
$\nabla \cdot \psi =0$, we have
\begin{align*}
\begin{aligned}
\int_{\bbr^3}u(t)\psi dx + \int^{t}_{0}\int_{\bbr^3}-u \cdot
\partial_t \psi + [\tilde{G}+ (u \otimes u)]:\nabla \psi
\,dxds =\int_{\bbr^3}u_0 \cdot \psi(x,\cdot) dx
\end{aligned}
\end{align*}
for a. e., $t\in[0,T]$. Here, $\tilde{G}$ is a weak limit of
$G[|Du_{m_k}|^2]Du_{m_k}$ in $L^{q'}((0,T_l)\times\bbr^3)$. Due to
the strong convergence \eqref{NNS-strong}, it follows that
\[
\tilde{G} = G[|Du|^2]Du \quad {\rm a.e.}\, \quad \mbox{in} \quad
\bbr^3\times (0, T_0).
\]
As $m\rightarrow\infty$, we conclude the existence of a solution of
\eqref{NNS}-\eqref{ini} in $\calX_l:=L^{\infty}([0, T_{l}];\;
H^{l}(\bbr^3))\cap L^{2}([0, T_{l}]; \; H^{l+1}(\bbr^3))$, $l\geq
3$. Next, we show uniqueness of the solution of
\eqref{NNS}-\eqref{ini} until the time $t\le T_l$. More precisely,
we prove uniqueness of weak and regular solutions, in case that
initial data are the same. Let $(u^1,p^1)$ be a weak solution and
$(u^2,p^2)$ a solution constructed above of the equation
\eqref{NNS}-\eqref{ini}. We consider the equation for
$\tilde{u}=u_{1}-u_{2}$ and $\tilde{p}=p_{1}-p_{2}$.
\begin{equation}\label{uniqu-100-10}
\begin{gathered}
\partial_t\tilde{u}-\nabla \cdot
(G[|Du_1|^2]Du_{1}-G[|Du_2|^2]Du_2) +(u_1\cdot
\nabla)\tilde{u}+(\tilde{u}\cdot \nabla)u_2+\nabla \tilde{p}=0,\\
\operatorname{div}\ \tilde{u}=0.
\end{gathered}
\end{equation}
Testing $\tilde{u}$ to the difference equation \eqref{uniqu-100-10}
and integrating it by parts, we have
\begin{equation}\label{unique-100}
\frac{1}{2}\frac{d}{dt}\|\tilde{u}\|^2_{L^2(\R^3)}+m_0\|\nabla
\tilde{u}\|^2_{L^2(\R^3)} \leq \int_{\R^3}|\tilde{u}||\nabla u_2|
|\tilde{u}|\,dx:=J_1,
\end{equation}
where we use the divergence free condition and Lemma \ref{r-400}.
Since $u_2 \in L^\infty(0,T_l; H^3(\R^3))\cap L^2(0,T_l;
H^4(\R^3))$, we can estimate the term $J_1$ as follows. Using
H\"{o}lder's inequality and Young's inequality, we have
\[
|J_1|\leq\|\tilde{u}\|^2_{L^2(\R^3)}\|\nabla
u_2\|_{L^{\infty}(\R^3)}.
\]
Thus, the estimate \eqref{unique-100} becomes
\[
\frac{d}{dt}\|\tilde{u}\|^2_{L^2(\R^3)}+m_0\|\nabla
\tilde{u}\|^2_{L^2(\R^3)}\leq
C\|\tilde{u}\|^2_{L^2(\R^3)}\|u_2\|_{H^{3}(\R^3)}.
\]
Due to Grownwall's inequality and $\tilde{u}(x,0)=0$, we conclude
that $\tilde{u}=0$, i. e., $u_1=u_2$. \qed
\end{pfthm1}

\section{Proof of corollary \ref{main-cor}}
For the proof of corollary \ref{main-cor}, we need the following
lemma, which is a simpler case of \cite[Lemma 2.2]{JS98}. For
clarity, we present its proof. For notational convention, we write
the average of $f$ on $E$ as $\aint_{E} f$, that is $\aint_{E}f
=\frac{1}{|E|}\int_{E} f$.

\begin{lemma}\label{holder-lem}
Let $\alpha\in (0,\frac{1}{2}]$. Suppose that $v\in L^\infty([0,T];
C^\alpha(\R^3))$ and $v_t \in L^\infty([0,T]; L^2(\R^3))$. Then $v
\in C_x^{\alpha}C_t^{\frac{\alpha}{2}}(\R^3\times [0,T])$.
\end{lemma}

\begin{proof}
To show $v \in C_x^{\alpha}C_t^{\frac{\alpha}{2}}(\R^3\times
(0,T_{0}))$, it suffices to show that $|v(x,t_1)-v(x,t_2)|\leq C
|t_1-t_2|^{\frac{\alpha}{2}}$. For $x\in \R^3$ and $\rho>0$ we
define
\[
v_{\rho}(x,t)=\aint_{B_{x,\rho}}v(y,t)\,dy.
\]
We note that
\begin{equation}\label{holdercon-100}
|v(x,t_1)-v(x,t_2)|\leq |v
(x,t_1)-v_{\rho}(x,t_1)|+|v_{\rho}(x,t_1)-v_{\rho}(x,t_2)|+|v_{\rho}(x,t_2)-v(x,t_2)|.
\end{equation}
By the hypothesis, first and third terms in \eqref{holdercon-100}
are estimated easily as
\[
|v(x,t_i)-v_{\rho}(x,t_i)|
=\Big|\aint_{B_{x,\rho}}v(x,t_i)-v(y,t_i)\,dy\Big|\leq
C\rho^{\alpha}, \quad i=1,2.
\]
For second term in \eqref{holdercon-100}, since $v_t \in
L^\infty([0,T]; L^2(\R^3))$, we obtain
\[
|v_{\rho}(x,t_2)-v_{\rho}(x,t_1)|=\Big|\aint_{B_{x,\rho}}v(x,t_2)-v(x,t_1)\,dy\Big|
=\Big|\aint_{B_{x,\rho}}\int_0^1\frac{d}{d\theta}v(x,\theta
t_2+(1-\theta) t_1)\,d\theta\,dy\Big|
\]
\[
\leq\aint_{B_{x,\rho}}\int_0^1|v_t(x,\theta t_2+(1-\theta)
t_1)||t_2-t_1|\,d\theta\,dy\leq
C\rho^{-\frac{3}{2}}\|v_t\|_{L^\infty(0,T;L^2(\R^3))}|t_2-t_1|.
\]
Putting $\rho=|t_2-t_1|^{1/\mu}$, we have
\[
|v(x,t_1)-v(x,t_2)| \leq C(\rho^{\alpha}+\rho^{-\frac{3}{2}+\mu}).
\]
Choosing $\mu=\alpha+\frac{3}{2}$, we see that
\[
|v(x,t_1)-v(x,t_2)|\leq C|t_2-t_1|^{1-\frac{3}{2\mu}},
\]
which implies $v\in C_x^{\alpha}C_t^{1-\frac{3}{2\mu}}$. It follows
from $\alpha\in (0,\frac{1}{2}]$ that $1-\frac{3}{2\mu}\ge
\alpha/2$. Therefore, we conclude that $v\in
C_x^{\alpha}C_t^{\frac{\alpha}{2}}$. This completes the proof.
\end{proof}

Next, we control mixed derivatives of $G[|Du|^2]Du$ for spatial and
temporal variables, in case that some mixed derivatives of $u$ are
bounded. Since the proof is similar to that of Lemma \ref{deriv-G},
likewise, the details are put off in the Appendix.
\begin{lemma}\label{deriv-G : time}
Let $l$ be a positive integer. If there exist constants
$\mathcal{A},~ \mathcal{B}>0$ such that
\begin{equation*}
\parallel\partial_x^{l-2m} \partial_t^{b} u \parallel_{L^\infty(0,T: L^2(\R^3))} \leq \mathcal{A},
\quad \parallel\partial_x \partial_t^{b} u \parallel_{L^\infty(0,T:
L^\infty(\R^3))} \leq \mathcal{B}
\end{equation*}
for any nonnegative integer $m$ and $b$ with $0\leq b \leq m \leq
\frac{l}{2}$, then
\begin{equation*}
\| \partial_x^{l-2m-1}\partial_t^m(G Du)
\|_{L^\infty(0,T:L^2(\R^3))} \leq C\mathcal{A}\big(\mathcal{B}+
\mathcal{B}^{l-m-1} \big).
\end{equation*}
\end{lemma}
\begin{proof}
See Appendix for the proof.
\end{proof}

Now we are ready to provide the proof of Corollary \ref{main-cor}.
\\
\begin{pfmain15}
First, we show the following statement:\\
Let $l$ and $m$ be integers with $0\leq m\leq \frac{l}{2}$ and
$l\geq 4$. If $u\in L^\infty(0,T: H^l(\R^3))$, then
\begin{equation}\label{eq-1 : cor}
\nabla_x^{l-2m}\partial_t^{b} u\in L^\infty(0,T:L^2(\R^3))
\end{equation}
for any integer $b$ satisfying $0\leq b \leq m $.

Indeed, using mathematical induction for $m$, we will prove \eqref{eq-1 : cor}.\\
$\bullet$ (Case $m=1$)\, Since $u\in L^\infty(0,T: H^l(\R^3))$, we
have
\begin{equation*}
\partial_x^{l-2} u\in L^\infty(0,T:L^2(\R^3)\cap L^\infty(\R^3)).
\end{equation*}
Next, we show
\begin{equation*}
\partial_x^{l-2} \partial_t u\in L^\infty(0,T:L^2(\R^3)).
\end{equation*}
From the equation \eqref{NNS}, we have
\begin{equation*}
\partial_x^{l-2} \partial_t u= - \partial_x^{l-2} \big(\nabla\cdot(G[|Du|^2] Du) + \nabla\cdot(u \otimes u) + \nabla p \big).
\end{equation*}
Hence, we are going to show
\begin{equation*}
 \partial_x^{l-1} \big(G[|Du|^2] Du\big), ~ \partial_x^{l-1} \big(u \otimes u) , ~ \partial_x^{l-1} p ~ \in  L^\infty(0,T:L^2(\R^3)).
\end{equation*}
Taking the divergence operator to \eqref{NNS}, we note that
\begin{equation*}
-\Delta p={\rm div}[{\rm div}(G[|Du|^2]Du)]+{\rm div}{\rm
div}(u\otimes u).
\end{equation*}
From the standard elliptic theory, it is enough to show
\begin{equation*}
 \partial_x^{l-1} \big(G[|Du|^2] Du\big), ~ \partial_x^{l-1} \big(u \otimes u)  ~ \in  L^\infty(0,T:L^2(\R^3)).
\end{equation*}
First, we note
\[
\| \partial_x^{l-1} \big(G[|Du|^2] Du\big)\|_{L^2}\leq \|G[|Du|^2]
\partial_x^{l-1} \big(Du\big)\|_{L^2} +
\sum_{\alpha=1}^{\alpha=l-1} \|\partial_x^\alpha G[|Du|^2]
~\partial_x^{l-1-\alpha}Du \|_{L^2}
\]
\[
\leq \|G[|Du|^2]\|_{L^\infty}\|\nabla^{l-1}Du \|_{L^2} + C \|G[
|Du|^2 ]\|_{L^\infty}\big( \|Du\|_{L^2} + \|Du\|_{L^2}^{l-1}\big)
\|\nabla^{l-1}Du \|_{L^2}
\]
\[
\leq C \|G[ |Du|^2 ]\|_{L^\infty} \big( 1 +
\|Du\|_{L^2(\R^3)}^{l-1}\big)\|\nabla^{l}u \|_{L^2(\R^3)},
\]
where we used \eqref{eq-3 : derive G} in the second inequality.

\noindent Next, we have
\[
\| \partial_x^{l-1} \big(u \otimes u \big)\|_{L^2} \leq
\sum_{\alpha=0}^{\alpha=l-1} \|\partial_x^\alpha u \otimes
\partial_x^{l-1-\alpha}u \|_{L^2}
\]
\[
\leq  \sum_{\alpha=0}^{\alpha=l-2} \|\partial_x^\alpha
u\|_{L^\infty} \| \partial_x^{l-1-\alpha}u \|_{L^2} + \|u
\|_{L^\infty}\|\partial_x^{l-1}u \|_{L^2}\leq C \|u \|_{H^l}^2.
\]

\noindent $\bullet$ (Case $m=n+1$)\,Assume \eqref{eq-1 : cor} holds
for the case $m=n $, that is
\begin{equation}\label{eq-2 : cor}
\nabla_x^{l-2n}\partial_t^{b} u\in L^\infty(0,T:L^2)
\end{equation}
for any integer $b$ such that $0\leq b \leq n$. We then show the
case $m=n+1$, that is,
\begin{equation}\label{eq-9 : cor}
\nabla_x^{l-2(n+1)}\partial_t^{b} u\in L^\infty(0,T:L^2), \qquad
0\leq b\leq n+1.
\end{equation}
Since the other cases $0\leq b\leq n$ can be shown similarly, we
only prove \eqref{eq-9 : cor} for $b=n+1$, which is
\begin{equation*}
\nabla_x^{l-2(n+1)}\partial_t^{n+1} u\in L^\infty(0,T:L^2).
\end{equation*}
\noindent Note that, similarly to the case $m=1$, once we have
\begin{equation}\label{eq-5 : cor}
\nabla_x^{l-2(n+1)+1}\partial_t^{n} \big(G[|Du|^2] Du\big), ~
\nabla_x^{l-2(n+1)+1}\partial_t^{n} \big(u \otimes u)  ~ \in
L^\infty(0,T:L^2(\R^3)),
\end{equation}
we conclude
\begin{equation}\label{eq-6 : cor}
\nabla_x^{l-2(n+1)+1}\partial_t^{n} p ~ \in L^\infty(0,T:L^2(\R^3)).
\end{equation}
Combining \eqref{eq-5 : cor} and \eqref{eq-6 : cor} with the
equation \eqref{NNS}, we get
\begin{equation*}\label{eq-7 : cor}
\nabla_x^{l-2(n+1)}\partial_t^{(n+1)} u\in L^\infty(0,T:L^2).
\end{equation*}
Therefore, it remains to prove \eqref{eq-5 : cor} to conclude the
proof. Let us prove \eqref{eq-5 : cor}. First of all, it follows
from Lemma \ref{deriv-G : time} that
\begin{equation*}
\begin{aligned}
 \| \partial_x^{l-2n-1}\partial_t^n \big(G[|Du|^2] Du\big)\|_{L^\infty(0,T:L^2(\R^3))}&\leq C\mathcal{A}(\mathcal{B}+\mathcal{B}^{l-n-1}),
 \end{aligned}
\end{equation*}
where
\begin{equation*}
\begin{aligned}
 \| \partial_x^{l-2n}\partial_t^bu\|_{L^\infty(0,T:L^2(\R^3))}\leq \mathcal{A}, \quad
\| \partial_x\partial_t^bu\|_{L^\infty(0,T:L^\infty(\R^3))}\leq
\mathcal{B}.
 \end{aligned}
\end{equation*}
Due to \eqref{eq-2 : cor}, we note that $\mathcal{A} , ~
\mathcal{B}<\infty$, which implies
\begin{equation*}
\begin{aligned}
 \| \partial_x^{l-2n-1}\partial_t^n \big(G[|Du|^2] Du\big)\|_{L^\infty(0,T:L^2(\R^3))}<\infty.
 \end{aligned}
\end{equation*}
Next, we can also observe from the assumption \eqref{eq-2 : cor}
that
\begin{equation*}
\begin{aligned}
 \| \partial_x^{l-2n-1}\partial_t^n \big(u \otimes u \big)\|_{L^2(\R^3)}
 <\infty.
 \end{aligned}
\end{equation*}
This completes the proof of \eqref{eq-1 : cor}.

Now, we finish the proof of the Corollary \ref{main-cor}. For some
integers $b$ and $m$ satisfying $0\leq b\leq m\leq \frac{l}{2}-1$,
we define
\begin{equation*}
v:=\partial_x^{l-2(m+1)}\partial_t^b u.
\end{equation*}
Then, via the estimate \eqref{eq-1 : cor}, we have
\begin{equation}\label{eq-3 : cor}
v \in L^\infty(0,T: H^2(\R^3)) \subset L^\infty(0,T:
C^{\frac{1}{2}}(\R^3)).
\end{equation}
Furthermore, we also have
\begin{equation}\label{eq-4 : cor}
\begin{aligned}
\partial_t v=\partial_x^{l-2(m+1)}\partial_t^{b+1} u\in L^\infty(0,T: L^2(\R^3))
\end{aligned}
\end{equation}
from the estimate \eqref{eq-1 : cor} due to the fact $0\leq b+1 \leq
m+1 \leq \frac{l}{2}$. We plug \eqref{eq-3 : cor} and \eqref{eq-4 :
cor} into Lemma \ref{holder-lem}, and conclude
\begin{equation*}
\begin{aligned}
\partial_x^{l-2(m+1)}\partial_t^{b} u\in C_x^{\frac{1}{2}}C_t^{\frac{1}{4}}(\R^3 \times [0,T]), \quad \forall ~0\leq b\leq m\leq
\frac{l}{2}-1,
\end{aligned}
\end{equation*}
which completes the proof of Corollary \ref{main-cor}.\qed
\end{pfmain15}

\section{Proof of Theorem \ref{main1}}
In this section, we present the proof of Theorem \ref{main1}. We
recall \eqref{eq-1 : H0}, \eqref{eq-3 : H1} and \eqref{eq-1 : H2}.
\begin{equation}\label{eq-1 : H0 small}
\frac{1}{2}\frac{d}{dt}||u||^2_{L^2(\R^3)}+\int_{\R^3}G[|Du|^2]
|Du|^2=0,
\end{equation}
\begin{equation}\label{eq-3 : H1 small}
\frac{1}{2}\frac{d}{dt}||\partial_{x_i}u||^2_{L^2(\R^3)}+\int_{\R^3} m_0|\partial_{x_i}Du|^2
\leq -\int_{\R^3}\partial_{x_i} \big((u\cdot \nabla) u \big )\cdot
\partial_{x_i} u,
\end{equation}
\[
\frac{1}{2}\frac{d}{dt}||\partial_{x_{j}}\partial_{x_{i}}u||^2_{L^2(\R^3)}+\int_{\R^3}m_0|\partial_{x_{j}}\partial_{x_{i}}D
u|^2 \leq C\|G[|Du|^2]\|_{L^\infty}(\|Du\|_{L^\infty} +
\|Du\|_{L^\infty}^2)\|\nabla^2 Du\|^2_{L^2}
\]
\begin{equation}\label{eq-1 : H2 small}
-\int_{\R^3}\partial_{x_{j}}\partial_{x_{i}}\big((u\cdot\nabla u)
\big)\cdot \partial_{x_{j}}\partial_{x_{i}} u.
\end{equation}
Via \eqref{eq-1 : H3}, \eqref{eq-2 : H3}, and \eqref{eq-5 : H3}, we
also remind that
\begin{equation}\label{eq-10 : H3 small}
\begin{aligned}
&\frac{1}{2}\frac{d}{dt}||\partial^3u||^2_{L^2(\R^3)} +
\int_{\R^3}G[|Du|^2]|\partial^3Du|^2\,dx
+ 2\int_{\R^3} G^{'}[|Du|^2]|Du:\partial^3 D u|^2\,dx\\
&\leq |I_{32}| + |I_{33}| - \int_{\R^3} E_3(D u:\partial^3D u)\,dx
-\int_{\R^3}\partial^3\big((u\cdot\nabla u) \big)\cdot
\partial^3 u\,dx,
\end{aligned}
\end{equation}
where $I_{32}$ and $I_{33}$ are given in \eqref{eq-2 : H3}. Now we
estimate each term on the right hand side of \eqref{eq-10 : H3
small} differently than we did in Theorem \ref{main}. We note first
that
\begin{equation*}
\begin{aligned}
|I_{32}|&\leq  \int_{\R^3}|\partial G[|Du|^2]| |\partial^2 Du||\partial^3 D u|\,dx\\
&\leq \|\partial G[|Du|^2]~ \partial^2 Du\|_{L^2}\|\nabla^3
Du\|_{L^2}\\
&\leq C \|G[ |Du|^2 ]\|_{L^\infty} \|Du\|_{L^\infty} \|\nabla^3
Du\|_{L^2}^2,
\end{aligned}
\end{equation*}
where we exploit \eqref{eq-3 : derive G} in the last inequality. For
$I_{33}$, again due to \eqref{eq-3 : derive G}, we have
\begin{equation*}
\begin{aligned}
|I_{33}|&\leq \int_{\R^3}|\partial^2 G[|Du|^2]| |\partial Du||\partial^3 D u|\,dx\\
&\leq \|\partial^2 G[|Du|^2]~ \partial Du\|_{L^2}\|\nabla^3 Du\|_{L^2}\\
&\leq C \|G[ |Du|^2 ]\|_{L^\infty} \big(\|Du\|_{L^\infty} +
\|Du\|_{L^\infty}^2\big)\|\nabla^3 Du\|_{L^2}^2.
\end{aligned}
\end{equation*}
Using again estimate \eqref{eq-3 : derive G}, we obtain
\begin{equation}\label{eq-13 : H3 small}
\begin{aligned}
\int_{\R^3}E_3(D u:\partial^3D u)\,dx
&  \leq \|E_3~ Du\|_{L^2}\|\nabla^3 Du\|_{L^2}\\
& \leq C \|G[ |Du|^2 ]\|_{L^\infty} \big(\|Du\|_{L^\infty} +
\|Du\|_{L^\infty}^3\big)\|\nabla^3 Du\|_{L^2}^2.
\end{aligned}
\end{equation}
Next, we estimate the terms caused by convection term. Using ${\rm
div}\,u=0$, we have
\begin{align}\label{eq-14 : H3 small}
\begin{aligned}
\int_{\bbr^3} \partial[(u\cdot \nabla)u]\cdot \partial u\,dx &=
\int_{\bbr^3} [(\partial u\cdot \nabla)u]\cdot \partial u\,dx\leq
C\|\nabla u \|_{L^\infty} \|\nabla u \|_{L^2}^2,
\end{aligned}
\end{align}
\begin{align}\label{eq-15 : H3 small}
\begin{aligned}
\int_{\bbr^3} \partial^2[(u\cdot \nabla)u]\cdot \partial^2 u\,dx
&=\int_{\bbr^3} [(\partial u\cdot \nabla)\partial u]\cdot \partial^2 u\,dx + \int_{\bbr^3} [(\partial^2 u\cdot \nabla)u]\cdot \partial^2 u\,dx\\
&\leq C \|\nabla u \|_{L^\infty} \|\nabla^2 u \|_{L^2}^2,
\end{aligned}
\end{align}
and
\begin{align}\label{eq-16 : H3 small}
\begin{aligned}
\int_{\bbr^3} \partial^3[(u\cdot \nabla)u]\cdot \partial^3 u\,dx &= \int_{\bbr^3} [(\partial^3 u\cdot \nabla)u]\cdot \partial^3 u\,dx
+ \int_{\bbr^3} [(\partial^2 u\cdot \nabla)\partial u]\cdot \partial^3 u\,dx \\
& \quad + \int_{\bbr^3} [(\partial u\cdot \nabla)\partial^2 u]\cdot \partial^3 u\,dx\\
&\leq C \|\nabla u \|_{L^\infty} \|\nabla^3 u \|_{L^2}^2 + C \|\nabla^2 u \|_{L^4}^2 \|\nabla^3 u \|_{L^2} \\
&\leq C \|\nabla u \|_{L^\infty} \|\nabla^3 u \|_{L^2}^2 .
\end{aligned}
\end{align}
Combining \eqref{eq-10 : H3 small}-\eqref{eq-13 : H3 small} and \eqref{eq-16 : H3 small}, we have
\begin{equation*}
\label{eq-2 : H3 small}
\begin{aligned}
&\frac{1}{2}\frac{d}{dt}||\partial^3u||^2_{L^2(\R^3)} +
\int_{\R^3}G[|Du|^2]|\partial^3Du|^2\,dx
+ 2\int_{\R^3} G^{'}[|Du|^2]|Du:\partial^3 D u|^2\,dx\\
&\leq  C \|G[ |Du|^2 ]\|_{L^\infty}
\big(\|Du\|^3_{L^\infty}+\|Du\|_{L^\infty} \big) \|\nabla^3
Du\|^{2}_{L^2} + C\|\nabla u\|_{L^\infty}\|\nabla^3 u\|^2_{L^2}.
\end{aligned}
\end{equation*}
Again, as before, we have
\begin{equation}\label{eq-1 : H3 small}
\begin{aligned}
&\frac{1}{2}\frac{d}{dt}||\partial^3u||^2_{L^2(\R^3)} +
\int_{\R^3}m_0 |\partial^3Du|^2\,dx
\\
&\leq  C \|G[ |Du|^2 ]\|_{L^\infty}
\big(\|Du\|^3_{L^\infty}+\|Du\|_{L^\infty} \big) \|\nabla^3
Du\|^{2}_{L^2} + C\|\nabla u\|_{L^\infty}\|\nabla^3 u\|^2_{L^2}.
\end{aligned}
\end{equation}
Finally, we combine \eqref{eq-1 : H0 small}, \eqref{eq-3 : H1
small}, \eqref{eq-1 : H2 small}, \eqref{eq-14 : H3 small},
\eqref{eq-15 : H3 small} and \eqref{eq-1 : H3 small} to conclude
\[
\frac{1}{2}\frac{d}{dt}||u||^2_{H^3(\R^3)}
+\int_{\R^3} m_0(|\nabla^3 Du|^2+|\nabla^2 Du|^2 + |\nabla
Du|^2 + | Du|^2)\,dx
\]
\[
\leq
C \|G[ |Du|^2 ]\|_{L^\infty} \big(\|Du\|^3_{L^\infty}+\|Du\|_{L^\infty}\big)\big(\|\nabla^3
Du\|^2_{L^2} +\|\nabla^2 Du\|^2_{L^2}\big )
\]
\[
+C\|\nabla u\|_{L^\infty} \big ( \|\nabla^3 u\|^2_{L^2} +\|\nabla^2
u\|^2_{L^2}+\|\nabla u\|^2_{L^2} \big)
\]
\[
\leq C
g(\|u|\|_{H^3})\big(\|u\|^3_{H^3}+\|u\|_{H^3}\big)\big(\|\nabla^3
Du\|^2_{L^2} +\|\nabla^2 Du\|^2_{L^2}\big )
\]
\begin{equation}\label{final-est-H3}
+C \|u\|_{H^3} \big ( \|\nabla^3 u\|^2_{L^2} +\|\nabla^2 u\|^2_{L^2}+\|\nabla u\|^2_{L^2} \big),
\end{equation}
where $g$ is a nondecreasing function defined in \eqref{eq-2 : H3 local}.
Since $\|u_0\|_{H^3} \leq \epsilon_0$, it follows from local
existence of solution that there exists a time $t^*>0$ such that
$\|u(t)\|_{H^3} \leq 2\epsilon_0$ for all $t\le t^*$. Thus, due to
estimate \eqref{final-est-H3}, we obtain for any $t\le t^*$
\begin{equation*}
\frac{1}{2}\frac{d}{dt}||u||^2_{H^3(\R^3)} +\int_{\R^3}(m_0
-2\epsilon_0 ) (|\nabla^3 Du|^2+|\nabla^2 Du|^2 + |\nabla Du|^2 + |
Du|^2)\,dx \leq 0,
\end{equation*}
which implies that after integrating it in time,
\[
\norm{u(t)}^2_{H^3(\R^3)}\le \norm{u_0}^2_{H^3(\R^3)}\le
\epsilon_0,\qquad t\le t^*.
\]
Repeating this procedure at $t=t^*$, we extend the solution in
$[t^*, 2t^*]$, which immediately implies that $T_3$ in Theorem
\ref{main} becomes infinity, i.e. $T_3=\infty$.

\subsection{Estimation of $\|u \|_{H^l}$, $l\ge 4$}

In case that $l=4$, via \eqref{eq-1 : H4}, \eqref{eq-2 : H4} and
\eqref{eq-7 : H4}, we remind that
\begin{equation}\label{eq-500-main-small}
\begin{aligned}
&\frac{1}{2}\frac{d}{dt}||\partial^4u||^2_{L^2(\R^3)}
 + \int_{\R^3}m_0|\partial^4D u|^2\,dx
 \\
\le & |I_{42}| + |I_{43}|+ |I_{44}|- \int_{\R^3}E_4(D u :\partial^4D
u)\,dx-\int_{\R^3}\partial^4\big((u\cdot\nabla u) \big) \cdot
\partial^4 u\,dx,
\end{aligned}
\end{equation}
where $I_{42},~I_{43}$ and $I_{44}$ are as defined in \eqref{eq-2 :
H4}.  We estimate the right hand side of \eqref{eq-500-main-small}.

\noindent For $I_{42}$, we exploit \eqref{eq-3 : derive G} to get
\begin{equation}\label{eq-200 : H4 small}
\begin{aligned}
|I_{42}|&\leq \int_{\R^3}|\partial G[|Du|^2]| |\partial^3 Du||\partial^4 D u|\,dx\\
& \leq C\|\partial G[|Du|^2]~ \partial^3 Du\|_{L^2}\|\nabla^4 Du\|_{L^2}\\
&\leq C \|G[ |Du|^2 ]\|_{L^\infty} \|Du\|_{L^\infty} \|\nabla^4
Du\|_{L^2}^2.
\end{aligned}
\end{equation}

\noindent Similarly, we have
\begin{equation}\label{eq-300-10-10 : H4 small}
\begin{aligned}
|I_{43}|&\leq \int_{\R^3}|\partial^2 G[|Du|^2]| |\partial^2 Du||\partial^4 D u|\,dx\\
& \leq C\|\partial^2 G[|Du|^2]~ \partial^2 Du\|_{L^2}\|\nabla^4 Du\|_{L^2}\\
&\leq C \|G[ |Du|^2 ]\|_{L^\infty} \big(\|Du\|_{L^\infty} +
\|Du\|_{L^\infty}^2 \big)\|\nabla^4 Du\|_{L^2}^2.
\end{aligned}
\end{equation}

\noindent For $I_{44}$, again due to \eqref{eq-3 : derive G}, we
obtain
\begin{equation}\label{eq-400 : H4 small}
\begin{aligned}
|I_{44}|&\leq \int_{\R^3}|\partial^3 G[|Du|^2]| |\partial Du||\partial^4 D u|\,dx\\
& \leq C\|\partial^3 G[|Du|^2]~ \partial Du\|_{L^2}\|\nabla^4 Du\|_{L^2}\\
&\leq C \|G[ |Du|^2 ]\|_{L^\infty} \big(\|Du\|_{L^\infty} +
\|Du\|_{L^\infty}^3 \big)\|\nabla^4 Du\|_{L^2}^2.
\end{aligned}
\end{equation}

\noindent It follows from \eqref{eq-3 : derive G} that
\begin{equation}\label{eq-500 : H4 small}
\begin{aligned}
&\int_{\R^3} |E_4(D u :\partial^4D u)|\,dx \leq \|E_4 Du\|_{L^2}\|\nabla^4 Du\|_{L^2}\\
&\leq C \|G[ |Du|^2 ]\|_{L^\infty} \big(\|Du\|_{L^\infty} +
\|Du\|_{L^\infty}^4 \big)\|\nabla^4 Du\|_{L^2}^2.
\end{aligned}
\end{equation}

\noindent  Lastly, for the convection term, using ${\rm div}\,u=0$,
we have
\begin{equation}\label{s-2-10:small}
\begin{aligned}
\int_{\bbr^3} \partial^{4}[(u\cdot \nabla)u]\cdot \partial^{4}u\,dx
&= \int_{\bbr^3} \sum_{\alpha=1}^4[(\partial^{\alpha} u\cdot \nabla)\partial^{4-\alpha} u]\cdot \partial^{4}u\,dx\\
&\leq C\big(\|\nabla u \|_{L^\infty} \|\nabla^{4}u \|_{L^2} + \|\nabla^{2}u \|_{L^\infty}\|\nabla^{3}u \|_{L^2} \big)\|\nabla^{4}u \|_{L^2}\\
&\leq C\big(\|\nabla u \|_{L^\infty} +  \|\nabla^{2}u \|_{L^\infty} \big) \big(\|\nabla^{4}u \|_{L^2}^2 + \|\nabla^{3}u \|_{L^2}^2 \big).
\end{aligned}
\end{equation}
In general, for $l \geq 4$, we note that
\begin{equation*}
\begin{aligned}
\int_{\bbr^3} \partial^{l}[(u\cdot \nabla)u]\cdot \partial^{l}u\,dx
& \leq C \Big(\sum_{\alpha =1}^{\lfloor \frac{l+1}{2} \rfloor} \|\nabla^\alpha u \|_{L^\infty} \|\nabla^{l+1-\alpha}u \|_{L^2} \Big )
\|\nabla^{4}u \|_{L^2}\\
& \leq C \Big(\sum_{\alpha =1}^{\lfloor \frac{l+1}{2} \rfloor} \|\nabla^\alpha u \|_{L^\infty} \Big )
\Big(\sum_{\alpha =\lfloor \frac{l}{2} \rfloor +1}^{l}  \|\nabla^{\alpha}u \|_{L^2}^2 \Big ).
\end{aligned}
\end{equation*}
Finally, we combine \eqref{final-est-H3}, \eqref{eq-500-main-small},
\eqref{eq-200 : H4 small}, \eqref{eq-300-10-10 : H4 small},
\eqref{eq-400 : H4 small}, \eqref{eq-500 : H4 small} and
\eqref{s-2-10:small} to conclude
\[
\frac{1}{2}\frac{d}{dt}||u||^2_{H^4(\R^3)}
+\int_{\R^3} m_0(|\nabla^4 Du|^2+|\nabla^3 Du|^2+|\nabla^2
Du|^2 + |\nabla Du|^2 + | Du|^2)\,dx
\]
\[
\leq
C\|G[|Du|^2]\|_{L^\infty}\big(\|Du\|^4_{L^\infty}+\|Du\|_{L^\infty}\big)
\big(\|\nabla^4 Du\|^2_{L^2}+\|\nabla^3 Du\|^2_{L^2}  +\|\nabla^2
Du\|^2_{L^2}\big )
\]
\[
+C\big(\|\nabla u\|_{L^\infty}^2 +\|\nabla u\|_{L^\infty} \big)\big ( \|\nabla^4 u\|^2_{L^2} +\|\nabla^3 u\|^2_{L^2} +\|\nabla^2 u\|^2_{L^2}+\|\nabla u\|^2_{L^2} \big)
\]
\begin{equation*}
\leq C \big( g(\|u|\|_{H^3}+1\big)
\big(\|u\|^4_{H^3}+\|u\|_{H^3}\big) \big(\|\nabla^4
Du\|^2_{L^2}+\|\nabla^3 Du\|^2_{L^2}  +\|\nabla^2 Du\|^2_{L^2}+
\|\nabla Du\|^2_{L^2}+ \| Du\|^2_{L^2}\big ).
\end{equation*}
Since $ \|u(t)\|_{H^3} \leq \epsilon_0 \ll 1 $ for $t\in[0,\infty)$, we have
\[
\frac{1}{2}\frac{d}{dt}||u||^2_{H^4(\R^3)} +\int_{\R^3} m_0
(|\nabla^4 Du|^2+|\nabla^3 Du|^2+|\nabla^2 Du|^2 + |\nabla Du|^2 + |
Du|^2)\,dx
\]
\[
\leq C\epsilon_0\big (\|\nabla^4 Du\|^2_{L^2} + \|\nabla^3
Du\|^{2}_{L^2} + \|\nabla^2 Du\|_{L^2}^2 + \|\nabla Du\|_{L^2}^2 +
\| Du\|_{L^2}^2\big ),
\]
and, therefore, we obtain
\[
\frac{1}{2}\frac{d}{dt}||u||^2_{H^4(\R^3)} +\int_{\R^3}  (m_0-C\epsilon_0)  (|\nabla^4 Du|^2+|\nabla^3 Du|^2+|\nabla^2 Du|^2 + |\nabla Du|^2 + | Du|^2)\,dx\\
\leq 0.
\]
This implies the global existence of solution in the class
$L^\infty(0,\infty : H^4)$.

For general $l> 4$, it follows from \eqref{eq-6 : I-l1} that
\[
\frac{1}{2}\frac{d}{dt}||\partial_{x_{\sigma(l)}}\partial_{x_{\sigma(i)}}\cdots\partial_{x_{\sigma(1)}}u||^2_{L^2(\R^3)}
+\int_{\R^3} m_0
|\partial_{x_{\sigma(l)}}\partial_{x_{\sigma(i)}}\cdots\partial_{x_{\sigma(1)}}
D u|^2\,dx
\]
\begin{equation}\label{eq-10 : Hl}
\leq |I_{l1}| +  |I_{l2}| +  |I_{l3}|,
\end{equation}
where $I_{l1}, I_{l2}$ and $I_{l3}$ are given in \eqref{eq-6 :
I-l1}. We then first estimate $I_{l1}$.
\begin{equation}\label{eq-1 : Hl}
\begin{aligned}
|I_{l1}|&\leq \sum_{\alpha=1}^{l-1}\int_{\R^3}|\partial^\alpha G[|Du|^2]| |\partial^{l-\alpha} D u||\partial^l D u| \,dx \\
& \leq \sum_{\alpha=1}^{l-1}\| \partial^\alpha G[|Du|^2] ~ \partial^{l-\alpha} D u\|_{L^2}  \|\nabla^l D u \|_{L^2} \\
& \leq C \|G[ |Du|^2 ]\|_{L^\infty} \sum_{\alpha=1}^{l-1}\big( \| Du\|_{L^\infty} + \| Du\|_{L^\infty}^\alpha \big)  \|\nabla^l D u \|_{L^2}^2 \\
&\leq C \|G[ |Du|^2 ]\|_{L^\infty} \big( \| Du\|_{L^\infty} + \| Du\|_{L^\infty}^{l-1} \big)  \|\nabla^l D u \|_{L^2}^2\\
&\leq Cg(\|u|\|_{H^3}) \big( \| u\|_{H^3} + \| u\|_{H^3}^{l-1} \big)
\|\nabla^l D u \|_{L^2}^2,
\end{aligned}
\end{equation}
where we exploit \eqref{eq-3 : derive G} in the third inequality.
Similarly, the term $|I_{l2}|$ is estimated
\begin{equation}\label{eq-2 : Hl}
\begin{aligned}
|I_{l2}|&\leq C\int_{\R^3}|E_{l}||D u| |\partial^l Du| \,dx \\
& \leq C\| E_{l} D u\|_{L^2}  \|\nabla^l D u \|_{L^2} \\
& \leq C \|G[ |Du|^2 ]\|_{L^\infty} \big( \| Du\|_{L^\infty} + \| Du\|_{L^\infty}^{l} \big)  \|\nabla^l D u \|_{L^2}^2\\
&\leq C g(\|u|\|_{H^3}) \big( \| u\|_{H^3} + \| u\|_{H^3}^{l} \big)
\|\nabla^l D u \|_{L^2}^2.
\end{aligned}
\end{equation}

\noindent Lastly, we estimate the convection term : Using divergence free condition
\begin{align*}
\begin{aligned}
\int_{\bbr^3} \partial^{\alpha}[(u\cdot \nabla)u]\cdot \partial^{\alpha}u\,dx
&=\sum_{k=1}^\alpha\int_{\bbr^3} [(\partial^{k}u\cdot \nabla)\partial^{\alpha-k}u]\cdot \partial^{\alpha}u\,dx\\
&\leq C\sum_{k=1}^\alpha \|\partial^{k}u \|_{L^p} \| \partial^{\alpha-k} \nabla u\|_{L^q} \|\partial^{\alpha}u \|_{L^2}\\
&\leq C \|\partial^{\alpha}  u \|_{L^2} \|  \nabla u\|_{L^\infty} \|\partial^{\alpha}u \|_{L^2},
\end{aligned}
\end{align*}
where we apply Gagliardo-Nirenberg to the third inequality. This gives us
\begin{align}\label{s-2}
\begin{aligned}
\sum_{\alpha=1}^l\int_{\R^3} \partial^{\alpha}[(u\cdot
\nabla)u]\cdot \partial^{\alpha}u\,dx &\leq C \|\nabla
u\|_{L^\infty}\sum_{\alpha=1}^l \|\partial^{\alpha}  u
\|_{L^2}^2\leq C \| u\|_{H^3}\sum_{\alpha=1}^l \|\partial^{\alpha} u
\|_{L^2}^2.
\end{aligned}
\end{align}
Plugging \eqref{eq-1 : Hl}, \eqref{eq-2 : Hl} and \eqref{s-2} into
\eqref{eq-10 : Hl}, since $ \|u(t)\|_{H^3} \leq \epsilon$ for
$t\in[0,\infty)$, we obtain
\begin{equation*}\label{eq-5 : Hl}
\frac{1}{2}\frac{d}{dt}||u||^2_{H^l(\R^3)} +\int_{\R^3}\big(m_0 -
\epsilon \big)(|\nabla^l Du|^2+|\nabla^{l-1} Du|^2+\cdots + |
Du|^2)\,dx \leq 0.
\end{equation*}
The above estimate implies that solution exists globally in time.
\qed

 \section*{Acknowledgments}
Kyungkeun Kang's work is partially supported by NRF-2017R1A2B4006484
and NRF-2015R1A5A1009350 and is also supported in part by the Yonsei
University Challenge of 2017. Hwa Kil Kim's work is supported by
NRF-2015R1D1A1A01056696 and NRF-2018R1D1A1B07049357. Jae-Myoung
Kim's work is supported by NRF-2015R1A5A1009350 and
NRF-2016R1D1A1B03930422.

\appendix
\section{}

In this Appendix, we provide the proofs of Lemma \ref{deriv-G} and
Lemma \ref{deriv-G : time}.

\begin{pflem2}
We note first that
\[
\partial_{x_i} G[|Du|^2] = 2 G^{'}[|Du|^2] (Du: \partial_{x_i} Du),\qquad i=1,2,3,
\]
From now on, we write $G[|Du|^2]$ and $G'[|Du|^2]$ as $G$ and $G'$,
respectively, unless any confusion is to be expected. For
convenience, we set $E_1=0$. Direct computations show that
\begin{equation}\label{eq-2 : appendix}
\partial_{x_j}\partial_{x_i} G = 2 G^{'} (Du: \partial_{x_j}\partial_{x_i} Du)  + 2 \big (\partial_{x_j}( G^{'}Du): \partial_{x_i} Du
\big),\qquad i, j=1,2,3.
\end{equation}
We define the second term of the righthand side in \eqref{eq-2 :
appendix} by $E_2$, namely
\begin{equation*}\label{eq-4 : appendix }
\begin{aligned}
E_2:= 2 \big (\partial_{x_j}( G^{'}Du): \partial_{x_i} Du \big).
\end{aligned}
\end{equation*}
Next, we define $E_l$, $l\ge 3$, inductively via multi-derivatives
of $G$.
\begin{equation*}\label{eq-3 : appendix}
\begin{aligned}
\partial_{x_k}\partial_{x_j}\partial_{x_i} G &= 2 G^{'} (Du: \partial_{x_k}\partial_{x_j}\partial_{x_i} Du)
+ 2 \big (\partial_{x_k}( G^{'}Du): \partial_{x_j}\partial_{x_i} Du \big) +\partial_{x_k} E_2 \\
&= 2 G^{'} (Du: \partial_{x_j}\partial_{x_i} Du)  + E_3,\qquad i,
j,k=1,2,3,
\end{aligned}
\end{equation*}
where
\begin{equation*}\label{eq-5 : appendix}
\begin{aligned}
E_3 &= 2 \big (\partial_{x_k}( G^{'}Du):
\partial_{x_j}\partial_{x_i} Du \big) +\partial_{x_k} E_2.
\end{aligned}
\end{equation*}
For $l\geq 4$ we define $E_l$ inductively as follows:
\begin{equation}\label{kkk-May1-10}
\partial_{x_{\sigma(l)}}\partial_{x_{\sigma(l-1)}} \cdots\partial_{x_{\sigma(1)}}G[|Du|^2]
= 2\big (G'[|Du|^2] Du:
\partial_{x_{\sigma(l)}}\partial_{x_{\sigma(l-1)}}
\cdots\partial_{x_{\sigma(1)}}D u \big) + E_l,
\end{equation}
where
\begin{equation*}
E_l=2\big(\partial (G^{'}Du): \partial^{l-1} Du \big) + \partial
E_{l-1}.
\end{equation*}
More precisely, we have
\[
E_l = 2 \partial_{x_{\sigma(l)}} G^{'}\big (Du:
\partial_{x_{\sigma(l-1)}}\cdots\partial_{x_{\sigma(1)}} Du \big)
+ 2 G^{'} \big (\partial_{x_{\sigma(l)}} Du:
\partial_{x_{{\sigma(l-1)}}}\cdots\partial_{x_{\sigma(1)}} Du \big)
+\partial_{x_{\sigma(l)}} E_{l-1}
\]
\[
= 2  G^{''} (Du : \partial_{x_{\sigma(l)}} Du)\big (Du:
\partial_{x_{\sigma(l-1)}}\cdots\partial_{x_{\sigma(1)}} Du \big)
\]
\begin{equation}\label{eq-7 : appendix }
+ 2 G^{'} \big (\partial_{x_{\sigma(l)}} Du:
\partial_{x_{\sigma(l-1)}}\cdots\partial_{x_{\sigma(1)}} Du \big)
+\partial_{x_{\sigma(l)}} E_{l-1}.
\end{equation}

Next, we estimate $E_2$ and $E_3$. From \eqref{eq-7 : appendix }, we
have
\begin{equation}\label{eq-8 : appendix }
\begin{aligned}
E_2
&\simeq G^{''} (Du: \partial Du)(Du: \partial Du) +   G^{'}  (\partial Du: \partial Du)\\
& \leq \big(| G^{''}| |Du|^2 + |G^{'}|\big) |\nabla Du|^2\leq
CG|\nabla Du|^2,
\end{aligned}
\end{equation}
where we used the properties \eqref{G : property} of $G$ to get the
last inequality. Combining \eqref{eq-7 : appendix } and \eqref{eq-8
: appendix }, we have
\begin{equation*}
\begin{aligned}
E_3& \simeq     G^{''}(Du: \partial Du)  (Du: \partial^2 Du) + G^{'} (\partial Du: \partial^2 Du) \\
&\quad +  \partial G^{''} (Du: \partial Du)(Du: \partial Du) + G^{''} (Du: \partial Du)[(\partial Du: \partial Du)  + (Du: \partial^2 Du)]\\
& \quad+  \partial G^{'}  (\partial Du: \partial Du) +   G^{'}  (\partial^2 Du: \partial Du) \\
& \simeq  G^{''}(Du: \partial Du)  (Du: \partial^2 Du) + G^{'} (\partial Du: \partial^2 Du)\\
& \quad+   G^{'''}(Du: \partial Du) (Du: \partial Du)(Du: \partial
Du) + G^{''} (Du: \partial Du)(\partial Du: \partial Du).
\end{aligned}
\end{equation*}
Let $P_n(G,Du)$ stand for the linear combination of
$G^{(k)}\big(Du\big)^l$ with $0\leq l\leq k\leq n,$ where $G^{(k)}$
is the $k-$th derivative of $G$. Using this notation, we rewrite
$E_3$ as follows
\begin{equation}\label{eq-14 : appendix }
\begin{aligned}
E_3& \simeq     P_3(G,Du)\big [(\partial Du)(\partial^2 Du)+
(\partial Du)^3 \big ] .
\end{aligned}
\end{equation}
Due to the property \eqref{G : property}, we note that
\[
|G^{(k)}[|Du|^2]||Du|^l \leq C |G^{(k-1)}[|Du|^2]||Du|^{l-1} \leq
\cdots \leq CG[|Du|^2]
\]
for any  $0\leq l\leq k$. This implies
\begin{equation}\label{eq-29 : appendix }
\begin{aligned}
|P_n(G,Du)|&\leq C G[|Du|^2], \qquad \forall ~ n\geq 1.
\end{aligned}
\end{equation}
Hence, we have
\begin{equation*}
|E_3|\leq C  G[|Du|^2]  \big (|\nabla Du||\nabla^2 Du|+   |\nabla
Du|^3 \big ).
\end{equation*}
This completes the proof of \eqref{eq-5 : derive G} in Lemma
\ref{deriv-G}.

It remains to prove \eqref{eq-3 : derive G}-\eqref{eq-4 : derive
G-v1} in Lemma \ref{deriv-G}. For notational convention, we denote
$k$-th order spatial derivative operators by $\partial^k $, unless
any confusion is to be expected. We then introduce $R_m(Du)$ as a
linear combination of
$$
(\partial^{a_1} Du)^{i_1} (\partial^{a_2} Du)^{i_2} \cdots
(\partial^{a_k} Du)^{i_k}, \qquad a_1,i_1,\cdots,a_k, i_k \in
\mathbb{N},
$$
 such that
$$a_1 i_1 + a_2 i_2 + \cdots + a_k i_k =m \qquad \mbox{and} \qquad 1\leq a_j \leq m-1, \quad \forall ~ j =1,\cdots, k. $$
Note that, for example, $R_3( Du)=(\partial Du)(\partial^2 Du)+
(\partial Du)^3$. We then rewrite \eqref{eq-14 : appendix } as
$$E_3\simeq     P_3(G,Du)R_3( Du).$$
In general, we will show that
\begin{equation}\label{eq-23 : appendix }
\begin{aligned}
E_n& \simeq     P_n(G,Du)R_n( Du), \qquad n \geq 1.
\end{aligned}
\end{equation}
$\bullet$ (Proof of \eqref{eq-23 : appendix })\,\, We prove
\eqref{eq-23 : appendix } by inductive argument. It is already shown
that \eqref{eq-23 : appendix } holds for $n=1,2,3$. Now we suppose
\eqref{eq-23 : appendix } is true for $n=m\ge 3$. It follows from
\eqref{eq-7 : appendix } that
\begin{equation}\label{eq-26 : appendix }
\begin{aligned}
E_{m+1}&= \big(G^{''}(Du)(\partial Du) + G^{'}\partial Du \big)(\partial^m Du) + \partial E_m\\
& \simeq \big(G^{''}(Du) + G^{'} \big)(\partial Du)(\partial^m Du)
 + \partial\big( P_m(G,Du) \big) R_m( Du)\\
 & \quad + P_m(G,Du) \partial \big(R_m( Du) \big).
\end{aligned}
\end{equation}
We first show that
\begin{equation}\label{eq-24 : appendix }
\partial\big( P_m(G,Du) \big)=P_{m+1}(G,Du) \partial Du.
\end{equation}
Indeed, for $0\leq l\leq k\leq  m$, we have
\begin{equation*}
\begin{aligned}
\partial\big (G^{(k)}(Du)^l\big) &=\partial\big (G^{(k)}\big)(Du)^l + G^{(k)}\partial\big ((Du)^l\big)\\
&=G^{(k+1)} (Du:\partial Du)(Du)^l + G^{(k)}(Du)^{l-1} \partial Du\\
&\simeq \big (G^{(k+1)}(Du)^{l+1} + G^{(k)}(Du)^{l-1}\big)\partial Du\\
& \simeq \big (G^{(k')}(Du)^{l'} \big)\partial Du,
\end{aligned}
\end{equation*}
where $0\leq l'\leq k'\leq  m+1$. This proves the identity
\eqref{eq-24 : appendix }.  Using the definition of $R_n$, the
direct computations show that
\begin{equation}\label{eq-25 : appendix }
\begin{aligned}
\partial R_m(Du) \simeq \partial R_{m+1}( Du).
\end{aligned}
\end{equation}
Plugging \eqref{eq-24 : appendix } and \eqref{eq-25 : appendix }
into \eqref{eq-26 : appendix }, we obtain
\begin{equation*}
\begin{aligned}
E_{m+1}& \simeq \big(G^{''}(Du) + G^{'} \big)(\partial Du)(\partial^m Du)  +  P_{m+1}(G,Du) (\partial Du) R_m(Du)\\
 & \quad + P_m(G,Du) R_{m+1}( Du) \\
 & \simeq P_{m+1}(G,Du) R_{m+1}( Du),
\end{aligned}
\end{equation*}
where we used that
\[
\big(G^{''}(Du) + G^{'} \big), ~ P_m(G,Du)\simeq P_{m+1}(G,Du),
\]
\begin{equation*}
(\partial Du)(\partial^m Du), ~(\partial Du) R_m( Du) \simeq
R_{m+1}( Du).
\end{equation*}
This completes the proof of \eqref{eq-23 : appendix }.\qed
\\
\\
Next, we will prove that for any $1\leq \alpha \leq l$
\begin{equation}\label{eq-28 : appendix }
\| E_{\alpha}~ \partial^{l-\alpha} Du\|_{L^2} \leq C  \|G[ |Du|^2 ]\|_{L^\infty} \| \nabla^l Du\|_{L^2} \big(\|Du\|_{L^\infty}+
\|Du\|_{L^\infty}^{\alpha} \big),
\end{equation}
and
\begin{equation}\label{eq-40 : appendix }
\| \partial^{\alpha}G~ \partial^{l-\alpha} Du\|_{L^2} \leq C  \|G[ |Du|^2 ]\|_{L^\infty} \| \nabla^l Du\|_{L^2} \big(\|Du\|_{L^\infty}+
\|Du\|_{L^\infty}^{\alpha} \big),
\end{equation}

$\bullet$ (Proofs of \eqref{eq-28 : appendix } and \eqref{eq-40 :
appendix })\,\, We note from \eqref{eq-29 : appendix } that
\begin{equation}\label{eq-60 : appendix }
\begin{aligned}
|P_n(G,Du)|&\leq C G[|Du|^2], 
\qquad
\forall ~ n\geq 1.
\end{aligned}
\end{equation}
Next, for some $a_1,i_1,\cdots, a_k, i_k \in \mathbb{N} $
such that $a_1 i_1 + a_2 i_2 + \cdots a_k i_k =\alpha $, we observe
that
\begin{equation}\label{eq-30 : appendix }
\begin{aligned}
&\|(\partial^{a_1} Du)^{i_1} (\partial^{a_2} Du)^{i_2} \cdots
(\partial^{a_k} Du)^{i_k} (\partial^{l-\alpha}Du)\|_{L^2}\\
&\leq C \|\partial^{a_1} Du \|_{L^{p_1 i_1}}^{i_1}\cdots
\|\partial^{a_k} Du \|_{L^{p_k i_k}}^{i_k} \|\partial^{l-\alpha}Du
\|_{L^q},
\end{aligned}
\end{equation}
where
\begin{equation}\label{eq-31 : appendix }
\frac{1}{p_1} + \cdots + \frac{1}{p_k} + \frac{1}{q} = \frac{1}{2}.
\end{equation}
With the aid of Gagliardo-Nirenberg inequality, we note that
\begin{equation}\label{eq-32 : appendix }
\begin{aligned}
\|\partial^{a_j} Du \|_{L^{p_j i_j}} &\leq C \|\partial^l Du\|_{L^2}^{\theta_j} \| Du\|_{L^\infty}^{1-\theta_j}, \qquad 1\leq j\leq k\\
\|\partial^{l-\alpha} Du \|_{L^q} &\leq C \|\partial^l
Du\|_{L^2}^{\theta_q} \| Du\|_{L^\infty}^{1-\theta_q},
\end{aligned}
\end{equation}
such that
\begin{equation}\label{eq-33 : appendix }
\begin{aligned}
\frac{1}{p_j i_j}&=\frac{a_j}{3} + \Big(\frac{1}{2}-\frac{l}{3} \Big)\theta_j, \qquad 1\leq j\leq k\\
\frac{1}{q}&=\frac{l-\alpha}{3} + \Big(\frac{1}{2}-\frac{l}{3}
\Big)\theta_q.
\end{aligned}
\end{equation}
It follows from \eqref{eq-33 : appendix } that $1 =
\sum_{j=1}^{k}\theta_j i_j + \theta_q$. Indeed,
\begin{equation}\label{eq-34 : appendix }
\begin{aligned}
\Big(\sum_{j=1}^{k}\frac{1}{p_j }\Big) + \frac{1}{q} &=
\sum_{j=1}^{k}\Big[\frac{a_j i_j}{3} + \Big(\frac{1}{2}-\frac{l}{3}
\Big)\theta_j i_j\Big]
+ \frac{l-\alpha}{3} + \Big(\frac{1}{2}-\frac{l}{3} \Big)\theta_q\\
\Longrightarrow ~~  \frac{1}{2} &= \frac{\alpha}{3} +
\Big(\frac{1}{2}-\frac{l}{3} \Big)\sum_{j=1}^{k}\theta_j i_j
+ \frac{l-\alpha}{3} + \Big(\frac{1}{2}-\frac{l}{3} \Big)\theta_q\\
\Longrightarrow ~~  \frac{1}{2}-\frac{l}{3} &=  \Big(\frac{1}{2}-\frac{l}{3} \Big)\Big[\sum_{j=1}^{k}\theta_j i_j + \theta_q\Big]\\
\Longrightarrow ~~ 1 &=  \sum_{j=1}^{k}\theta_j i_j + \theta_q,
\end{aligned}
\end{equation}
where we used $ \sum_{j=1}^{k}a_j i_j = \alpha$. We plug
\eqref{eq-32 : appendix } and \eqref{eq-34 : appendix } into
\eqref{eq-30 : appendix } to get
\begin{equation*}
\begin{aligned}
&\|(\partial^{a_1} Du)^{i_1} (\partial^{a_2} Du)^{i_2} \cdots (\partial^{a_k} Du)^{i_k} (\partial^{l-\alpha}Du)\|_{L^2}\\
&\leq C  \|\partial^l Du\|_{L^2}^{(\theta_1 i_1 +\cdots +\theta_k
i_k + \theta_q)}
\| Du\|_{L^\infty}^{\big[(1-\theta_1)i_1 +\cdots +(1-\theta_k)i_k + 1-\theta_q \big]}\\
& \leq C \|\partial^l Du\|_{L^2} \| Du\|_{L^\infty}^{(i_1 +\cdots
+i_k)},
\end{aligned}
\end{equation*}
which immediately implies that
\begin{equation}\label{eq-36 : appendix }
\begin{aligned}
\|R_\alpha( Du) (\partial^{l-\alpha}Du)\|_{L^2}\leq C \|\partial^l
Du\|_{L^2}\big( \| Du\|_{L^\infty}^{1}+\cdots +\|
Du\|_{L^\infty}^{\alpha} \big).
\end{aligned}
\end{equation}
We combine \eqref{eq-23 : appendix }, \eqref{eq-60 : appendix }, and
\eqref{eq-36 : appendix } to conclude \eqref{eq-28 : appendix }.

To get \eqref{eq-40 : appendix }, we first note
\begin{equation*}
\begin{aligned}
\|G^{'}[|Du|^2]( Du :\partial^\alpha Du ) \partial^{l-\alpha}Du
\|_{L^2}
&\leq \|G^{'}[|Du|^2] Du \|_{L^\infty}\|\partial^\alpha Du  \partial^{l-\alpha}Du \|_{L^2}\\
&\leq C\|G[ |Du|^2 ]\|_{L^\infty}  \|\partial^\alpha Du  \|_{L^p}\|
\partial^{l-\alpha}Du \|_{L^q},
\end{aligned}
\end{equation*}
where $\frac{1}{p}+\frac{1}{q}=\frac{1}{2}$. Using
Gagliardo-Nirenberg inequality, we have
\begin{equation*}
\begin{aligned}
\|\partial^\alpha Du  \|_{L^p}\|  \partial^{l-\alpha}Du \|_{L^q}
&\leq C \|\partial^l Du  \|_{L^2}^{\theta_p}\|Du \|_{L^\infty}^{1-\theta_p}\|\partial^l Du  \|_{L^2}^{\theta_q}\|Du \|_{L^\infty}^{1-\theta_q}\\
&\leq C \|\partial^l Du  \|_{L^2}^{\theta_p + \theta_q}\|Du
\|_{L^\infty}^{2-(\theta_p+\theta_q)},
\end{aligned}
\end{equation*}
with
\begin{equation}\label{eq-41 : appendix }
\begin{aligned}
\frac{1}{p}=\frac{\alpha}{3}+
\Big(\frac{1}{2}-\frac{l}{3}\Big)\theta_p \quad \mbox{and} \quad
\frac{1}{q}=\frac{l-\alpha}{3} +
\Big(\frac{1}{2}-\frac{l}{3}\Big)\theta_q.
\end{aligned}
\end{equation}
Due to \eqref{eq-41 : appendix } and
$\frac{1}{p}+\frac{1}{q}=\frac{1}{2}$, we get $\theta_p + \theta_q
=1$. Hence, we have
\begin{equation}\label{eq-42 : appendix }
\begin{aligned}
\|G^{'}[|Du|^2]( Du :\partial^\alpha Du ) \partial^{l-\alpha}Du
\|_{L^2} &\leq C \|G[ |Du|^2 ]\|_{L^\infty} \|\partial^l Du  \|_{L^2}\|Du
\|_{L^\infty}.
\end{aligned}
\end{equation}
We combine \eqref{kkk-May1-10}, \eqref{eq-28 : appendix } and
\eqref{eq-42 : appendix } to conclude
\begin{equation*}
\begin{aligned}
\|\partial^\alpha G[|Du|^2] \partial^{l-\alpha} Du\|_{L^2}
&\leq \|G^{'}[|Du|^2]( Du :\partial^\alpha Du ) \partial^{l-\alpha}Du \|_{L^2}+ \|E_\alpha \partial^{l-\alpha}Du \|_{L^2}\\
&\leq C \|G[ |Du|^2 ]\|_{L^\infty} \|\nabla^l Du  \|_{L^2}\big(\|Du
\|_{L^\infty} + \|Du \|_{L^\infty}^\alpha \big),
\end{aligned}
\end{equation*}
which gives us \eqref{eq-40 : appendix }.\qed
\\
\\
Finally, we show that for $l\geq 4$ there exist some $\beta_1>1$ and
$\beta_2,~\beta_3>0$ such that
\begin{equation}\label{eq-37 : appendix }
\| E_{\alpha}~ \partial^{l-\alpha} Du\|_{L^2} \leq C  \|G[ |Du|^2 ]\|_{L^\infty} \| \nabla^{l-1} Du\|_{L^2}^{\beta_1}
\|Du\|_{L^\infty}^{\beta_2},\qquad 1\leq \alpha \leq l,
\end{equation}
\begin{equation}\label{eq-44 : appendix }
\|\partial^\alpha G~ \partial^{l-\alpha} Du\|_{L^2} \leq C  \|G[ |Du|^2 ]\|_{L^\infty} \| \nabla^{l-1}
Du\|_{L^2}^{\beta_1}\|Du\|_{L^\infty}^{\beta_3},\quad 1\leq \alpha
\leq l-1.
\end{equation}
$\bullet$  (Proofs of \eqref{eq-37 : appendix } and \eqref{eq-44 :
appendix })\,\, Proofs of \eqref{eq-37 : appendix } and \eqref{eq-44
: appendix } are exactly same as those of \eqref{eq-28 : appendix }
and \eqref{eq-40 : appendix } except for following inequalities:
\begin{equation*}
\begin{aligned}
\|\partial^{a_j} Du \|_{L^{p_j i_j}} &\leq C \|\partial^{l-1} Du\|_{L^2}^{\theta_j} \| Du\|_{L^\infty}^{1-\theta_j}, \qquad 1\leq j\leq k\\
\|\partial^{l-\alpha} Du \|_{L^q} &\leq C \|\partial^{l-1}
Du\|_{L^2}^{\theta_q} \| Du\|_{L^\infty}^{1-\theta_q},
\end{aligned}
\end{equation*}
and
\begin{equation}\label{eq-38 : appendix }
\begin{aligned}
\frac{1}{p_j i_j}&=\frac{a_j}{3} + \Big(\frac{1}{2}-\frac{l-1}{3} \Big)\theta_j, \qquad 1\leq j\leq k\\
\frac{1}{q}&=\frac{l-\alpha}{3} + \Big(\frac{1}{2}-\frac{l-1}{3}
\Big)\theta_q,
\end{aligned}
\end{equation}
instead of \eqref{eq-32 : appendix } and \eqref{eq-33 : appendix }.
We combine \eqref{eq-31 : appendix } and \eqref{eq-38 : appendix }
to get
\begin{equation*}
\begin{aligned}
\sum_{j=1}^k \frac{1}{p_j}+ \frac{1}{q}=\frac{\sum_{j=1}^k a_ji_j +
l-\alpha}{3}
+\Big(\frac{1}{2}-\frac{l-1}{3} \Big)\sum_{j=1}^k\theta_j i_j +\theta q\\
\Longrightarrow \sum_{j=1}^k\theta_j i_j +\theta q =
\frac{2l-3}{2l-5} :=\beta_1.
\end{aligned}
\end{equation*}
Hence, we have
\begin{equation}\label{eq-45 : appendix }
\| E_{\alpha}~ \partial^{l-\alpha} Du\|_{L^2} \leq C  \|G[ |Du|^2 ]\|_{L^\infty} \| \nabla^{l-1} Du\|_{L^2}^{\beta_1}
\|Du\|_{L^\infty}^{1+i_1+\cdots+i_k -\beta_1}.
\end{equation}
Similarly, if $1\le \alpha\le l-1$, we have
\begin{equation}\label{eq-43 : appendix }
\begin{aligned}
\|G^{'}[|Du|^2]( Du :\partial^\alpha Du ) \partial^{l-\alpha}Du
\|_{L^2} &\leq C \|G[ |Du|^2 ]\|_{L^\infty} \|\partial^{l-1} Du
\|_{L^2}^{\beta_1}\|Du \|_{L^\infty}^{2-\beta_1}.
\end{aligned}
\end{equation}
Since $l\geq 4$, we have
$$1 < \beta_1 <2 \qquad \mbox{and}\qquad 0< 1+i_1+\cdots+i_k -\beta_1, ~ 2-\beta_1 <\alpha.$$
We plug this into \eqref{eq-45 : appendix } and \eqref{eq-43 :
appendix } to conclude the proof. \qed
\end{pflem2}
\\

Next we provide the proof of Lemma \ref{deriv-G : time}.

\begin{pflem4}
We note that
\begin{equation}\label{eq-1 : derive G - time}
\begin{aligned}
\partial_x^{l-2m-1}\partial_t^m(G Du)
=\sum_{\alpha=0}^{l-2m-1}\sum_{\beta=0}^{m}
\big(\partial_x^{\alpha}\partial_t^{\beta} G \big)~
\big(\partial_x^{\tilde{l}-\alpha}\partial_t^{m-\beta} Du \big),
\end{aligned}
\end{equation}
where $\tilde{l}:= l-2m-1$. Similar to the proof of Lemma
\ref{deriv-G}, we have
\begin{equation*}
\begin{aligned}
 \partial_x^{\alpha}\partial_t^{\beta} G =2\big(G^{'}Du : \partial_x^{\alpha}\partial_t^{\beta} Du\big) +
 E_{\alpha+\beta},
\end{aligned}
\end{equation*}
where
\begin{equation}\label{eq-3 : derive G - time}
\begin{aligned}
  E_{\alpha + \beta}\simeq P_{\alpha + \beta}\big(G, Du \big) \tilde{R}_{\alpha + \beta}\big( Du \big).
\end{aligned}
\end{equation}
Here, $P_{\alpha + \beta}\big(G, Du \big)$ is a linear combination
of $G^{(i)}(Du)^j$ with $0\leq i\leq j \leq  \alpha + \beta$ as
before and $\tilde{R}_{\alpha + \beta}\big( Du \big)$ has a form
such that
\begin{equation}\label{eq-4 : derive G - time}
\begin{aligned}
  \tilde{R}_{\alpha + \beta}\big( Du \big)=\big(\partial_x^{a_1}\partial_t^{b_1} Du\big)^{n_1}\cdots \big(\partial_x^{a_k}\partial_t^{b_k} Du\big)^{n_k}
\end{aligned}
\end{equation}
with $a_1 n_1 \cdots + a_k n_k=\alpha$ and $b_1n_1+ \cdots +
b_kn_k=\beta$.
\begin{equation}\label{eq-5 : derive G - time}
\begin{aligned}
 \parallel \big(\partial_x^{\alpha}\partial_t^{\beta} G \big)~
\big(\partial_x^{\tilde{l}-\alpha}\partial_t^{m-\beta} Du \big )
\parallel_{L^2(\R^3)} \leq \parallel \big(G^{'}Du :
\partial_x^{\alpha}\partial_t^{\beta} Du\big)
\big(\partial_x^{\tilde{l}-\alpha}\partial_t^{m-\beta} Du \big )
\parallel_{L^2(\R^3)}\\
+ \parallel  E_{\alpha +
\beta}\big(\partial_x^{\tilde{l}-\alpha}\partial_t^{m-\beta} Du \big
) \parallel_{L^2(\R^3)} := I + II.
\end{aligned}
\end{equation}
Let us first estimate $II$. We combine \eqref{eq-3 : derive G -
time} and \eqref{eq-4 : derive G - time}, and use the property
\eqref{eq-29 : appendix } to get
\begin{equation}\label{eq-6 : derive G - time}
II \leq C \|G[ |Du|^2 ]\|_{L^\infty}
\parallel\partial_x^{a_1}\partial_t^{b_1} Du\parallel_{L^{p_1n_1}}^{n_1}\cdots \parallel\partial_x^{a_k}\partial_t^{b_k} Du\parallel_{L^{p_kn_k}}^{n_k}
\parallel \partial_x^{\tilde{l}-\alpha}\partial_t^{m-\beta} Du \parallel_{L^{q}},
\end{equation}
where $\frac{1}{p_1} +\cdots \frac{1}{p_k}
+\frac{1}{q}=\frac{1}{2}$. We exploit Gagliardo-Nirenberg as in
\eqref{eq-32 : appendix } and have
\begin{equation*}
\begin{aligned}
\|\partial_x^{a_j} \partial_t^{b_j} Du \|_{L^{p_j n_j }}
&\leq C \|\partial_x^{\tilde{l}} (\partial_t^{b_j} Du)\|_{L^2}^{\theta_j} \|\partial_t^{b_j} Du\|_{L^\infty}^{1-\theta_j}, \qquad 1\leq j\leq k\\
\|\partial_x^{\tilde{l}-\alpha}\partial_t^{m-\beta} Du\|_{L^q} &\leq
C \|\partial_x^{\tilde{l}}(\partial_t^{m-\beta}
Du)\|_{L^2}^{\theta_q} \|\partial_t^{m-\beta}
Du\|_{L^\infty}^{1-\theta_q},
\end{aligned}
\end{equation*}
where
\begin{equation}\label{eq-8 : derive G - time}
\begin{aligned}
\frac{1}{p_j n_j}&=\frac{a_j}{3} + \Big(\frac{1}{2}-\frac{\tilde{l}}{3} \Big)\theta_j, \qquad 1\leq j\leq k\\
\frac{1}{q}&=\frac{\tilde{l}-\alpha}{3} +
\Big(\frac{1}{2}-\frac{\tilde{l}}{3} \Big)\theta_q.
\end{aligned}
\end{equation}
  Note that $\tilde{l}=l-2m-1$ and $b_j \leq m,~$ for all $j=1,\dots,k$. Hence, we have
 \begin{equation}\label{eq-9 : derive G - time}
\begin{aligned}
\|\partial_x^{a_j} \partial_t^{b_j} Du \|_{L^{p_j n_j }}^{n_j}
&\leq C \|\partial_x^{l-2m} (\partial_t^{b_j} Du)\|_{L^2}^{\theta_j n_j} \|\partial_t^{b_j} Du\|_{L^\infty}^{(1-\theta_j)n_j}\\
&\leq C \mathcal{A}^{\theta_j n_j}\mathcal{B}^{(1-\theta_j)n_j},
\qquad 1\leq j\leq k.
\end{aligned}
\end{equation}
Similarly, we obtain
 \begin{equation}\label{eq-10 : derive G - time}
\begin{aligned}
\|\partial_x^{\tilde{l}-\alpha}\partial_t^{m-\beta} Du\|_{L^q} &\leq
C \mathcal{A}^{\theta_q}\mathcal{B}^{1-\theta_q}.
\end{aligned}
\end{equation}
We combine \eqref{eq-6 : derive G - time}, \eqref{eq-9 : derive G -
time} and \eqref{eq-10 : derive G - time} to get
\begin{equation*}
\begin{aligned}
 II &\leq C \|G[ |Du|^2 ]\|_{L^\infty} \mathcal{A}^{\theta_1n_1+\cdots+\theta_k n_k+ \theta_q}\mathcal{B}^{(1-\theta_1)n_1
 +\cdots+(1-\theta_k)n_k+(1-\theta_q)}.
\end{aligned}
\end{equation*}
We plug $a_1n_1+\cdots + a_kn_k =\alpha$ into \eqref{eq-8 : derive G
- time} and get $\theta_1n_1+\cdots +\theta_k n_k+ \theta_q =1$.
Hence, we finally have
\begin{equation}\label{eq-12 : derive G - time}
\begin{aligned}
 II &\leq C \|G[ |Du|^2 ]\|_{L^\infty} \mathcal{A}\mathcal{B}^{n_1+\cdots+n_k}.
\end{aligned}
\end{equation}
Similarly, we have estimates for $I$
\begin{equation}\label{eq-13 : derive G - time}
\begin{aligned}
 I &\leq C \|G[ |Du|^2 ]\|_{L^\infty}  \parallel
 \big(\partial_x^{\alpha}\partial_t^{\beta} Du\big) \big(\partial_x^{\tilde{l}-\alpha}\partial_t^{m-\beta} Du \big ) \parallel_{L^2(\R^3)}\\
 & \leq C \|G[ |Du|^2 ]\|_{L^\infty} \mathcal{A}\mathcal{B}.
\end{aligned}
\end{equation}
We first note
$$n_1+\cdots+n_k \leq \alpha +\beta \leq l-m-1,$$
 and combine \eqref{eq-1 : derive G - time}, \eqref{eq-5 : derive G - time}, \eqref{eq-12 : derive G - time} and \eqref{eq-13 : derive G - time}
 to conclude
\begin{equation*}
\begin{aligned}
\parallel \partial_x^{l-2m-1}\partial_t^m(G Du)\parallel_{L^2(\R^3)}
\leq C \|G[ |Du|^2 ]\|_{L^\infty} \mathcal{A}\big(\mathcal{B}+
\mathcal{B}^{l-m-1} \big).
\end{aligned}
\end{equation*}
This completes the proof. \qed
\end{pflem4}

\setcounter{equation}{0}


\begin{thebibliography}{00}

\bibitem{A94}  H. Amann, Stability of the rest state of a viscous incompressible fluid. Arch. Rational Mech. Anal.  126  (1994)
231--242.

\bibitem{AM74} G. Astarita, G. Marrucci, Principles of non-Newtonian fluid
mechanics, McGraw-Hill, London, New York, 1974



\bibitem{BW} H.-O. Bae, J. Wolf,
 Existence of strong solutions to the equations of unsteady motion of shear thickening incompressible fluids.
 Nonlinear Anal. Real World Appl. 23 (2015) 160--182.
%
%
%
%
%
\bibitem{B09-1} H. Beir\~{a}o da Veiga, Navier-Stokes equations with shear-thickening
viscosity. Regularity up to the boundary. J. Math. Fluid Mech. 11
(2009) 233--257.

%
%
\bibitem{B11} H. Beir\~{a}o da Veiga, P. Kaplick\'{y}, M. R{\aa}\v{z}i\v{c}ka,
Boundary regularity of shear thickening flows. J. Math. Fluid Mech.
13 (2011) 387--404.

\bibitem{B-D-R10} L.C. Berselli, L. Diening, M. R\r{u}\v{z}\`{i}\v{c}ka, Existence of strong solutions for incompressible fluids with shear
dependent viscosities. J. Math. Fluid Mech. 12 (2010) 101--132.

\bibitem{B} G. Bohme,
    Non-Newtonian fluid mechanics,
    North-Holland Series in Applied Mathematics and Mechanics,
    1987.

\bibitem{B-P07} D. Bothe, J. Pruss, $L^p$-theory for a class of non-Newtonian fluids,
SIAM J. Math. Anal. 39 (2007) 379--421.




\bibitem{D-R-W10} L. Diening, M. R\r{u}\v{z}\`{i}\v{c}ka,
J. Wolf, Existence of weak solutions for unsteady motions of
generalized Newtonian fluids. Ann. Sc. Norm. Super. Pisa Cl. Sci. 5
(2010) 1--46.


\bibitem{JS98} O. John, J. Star\'{a}, On the regularity of weak solutions to
parabolic systems in two spatial dimensions. Comm. Partial
Differential Equations 23 (1998) 1159--1170.

%

%
%

\bibitem{K-M-S02} P. Kaplicky, J. Malek, J. Stara, Global-in-time H\"{o}lder continuity of the velocity gradients for
fluids with shear-dependent viscosities. NoDEA Nonlinear
Differential Equations Appl. 9 (2002) 175--195.


\bibitem{K05}  P. Kaplicky, Regularity of flows of a non-Newtonian fluid subject to
Dirichlet boundary conditions. Z. Anal. Anwendungen  24 (2005)
467--486.

\bibitem{L67} O. A. Ladyzhenskaya,
   New equations for the description of the motions of viscous
   incompressible fluids, and global solvability for their boundary
   value problems, Trudy Mat. Inst. Steklov. 102 (1967)
   85--104.

\bibitem{L69} O. A. Ladyzhenskaya, The mathematical theory of viscous incompressible flow,
    Gordon and Breach, New York, 2nd edition 1969.

\bibitem{Lion69} J.-L. Lions,
Quelques m\'ethodes de r\'esolution des probl\`emes aux limites non
lin\'eaires. (French) Dunod, Paris, 1969.
%



\bibitem{M-N-R-R} J. M\'alek, M. Ne\v cas, M. Rokyta, M. R$\stackrel{\circ}{\textrm u}$\v zi\v
                       cka, Weak and Measure-valued Solutions to Evolutionary PDEs, Chapman \& Hall
                       1996.

\bibitem{M-R05} J. M\'alek, K.R. Rajagopal, Mathematical issues concerning the
Navier--Stokes equations and some of its generalizations, in:
Evolutionary Equations, vol. II, in: Handb. Differ. Equ.,
Elsevier/North-Holland, Amsterdam, 2005, pp. 371--459.


\bibitem{P}  M. Pokorn\'{y}, Cauchy problem for the non-Newtonian viscous incompressible fluid.
 Appl. Math. 41 (1996) 169--201.


\bibitem{W} J. Wolf,
   Existence of weak solutions to the equations of non-stationary
   motion of non-Newtonian fluids with shear rate dependent
   viscosity, J. Math. Fluid Mech. 9 (2007) 104--138.


\end{thebibliography}
\end{document}